\documentclass[a4paper,11pt]{article}

\usepackage{a4wide}
\usepackage[utf8]{inputenc}
\usepackage[T1]{fontenc}
\usepackage{amsmath,amsthm,amsfonts,bbm}
\usepackage{graphicx}
\usepackage{caption}
\usepackage[dvipsnames]{xcolor}
\usepackage{tikz}
\usepackage{mathtools}

\newtheorem{theorem}{Theorem}

\newtheorem{assumption}[theorem]{Assumption}
\newtheorem{thm}{Theorem}[section]
\newtheorem{cor}[thm]{Corollary}
\newtheorem{lem}[thm]{Lemma}
\newtheorem{rmk}[thm]{Remark}
\newtheorem{prop}{Proposition}

\newcommand{\N}{\mathbb N}
\newcommand{\Z}{\mathbb Z}
\newcommand{\R}{\mathbb R}

\newcommand{\intT}{\int\limits_{0}^T}

\newcommand{\omsig}{\Omega_{\sigma}}

\usepackage{diagbox}

\newcommand{\cu}{\operatorname{curl}}
\newcommand{\di}{\operatorname{div}}

\newcommand{\trt}{\gamma_\mathrm{t}}

\DeclareMathOperator*{\essinf}{ess\,inf}
\DeclareMathOperator*{\esssup}{ess\,sup}

\renewcommand{\vec}[1]{\underline{#1}}
\newcommand{\abs}[1]{\left\lvert{#1}\right\rvert}
\newcommand{\norm}[1]{{\left\lVert{#1}\right\rVert}}

\definecolor{dgreen}{rgb}{0,.6,0}

\usepackage{xparse}
\NewDocumentCommand{\tens}{t_}
{%
	\IfBooleanTF{#1}
	{\tensop}
	{\otimes}%
}

\begin{document}

\title{Space-Time FEM for the Vectorial Wave Equation under Consideration of Ohm's Law}
\author{Julia I.M.~Hauser$^1$}
\date{
        $^1$Joanneum Research, \\
        Institute for Telecommunications, Navigation and Signal Processing,\\
         Graz, Austria \\[1mm]
        {\tt julia.hauser@joanneum.at}  \\
      }
      
\maketitle

\begin{abstract}
  The ability to deal with complex geometries and to go to higher orders is the main advantage of space-time finite element methods. Therefore, we want to develop a solid background from which we can construct appropriate space-time methods. In this paper, we will treat time as another space direction, which is the main idea of space-time methods. First, we will briefly discuss how exactly the vectorial wave equation is derived from Maxwell's equations in a space-time structure, taking into account Ohm's law. Then we will derive a space-time variational formulation for the vectorial wave equation using different trial and test spaces. This paper has two main goals. First, we prove unique solvability for the resulting Galerkin--Petrov variational formulation. Second, we analyze the discrete equivalent of the equation in a tensor product and show conditional stability, i.e. a CFL condition.  \\
  Understanding the vectorial wave equation and the corresponding space-time finite element methods is crucial for improving the existing theory of Maxwell's equations and paves the way to computations of more complicated electromagnetic problems. 
\end{abstract}

\section{Introduction}	
\subsection{Derivation of the vectorial wave equation}
Before we discuss the vectorial wave equation, let us derive the equation first. For that purpose, we need to take a closer look at Maxwell's equations which are used to model electromagnetic problems and are the foundation of classical electromagnetism. Maxwell's equations are given by a system of partial differential equations, namely
\begin{subequations}\label{eq:MaxwellSys}
	\begin{equation}\label{eqn:Max1}
		\cu_x\ \vec E = -\partial_t \vec B,
	\end{equation}
	\begin{equation}\label{eqn:Max2}
		\cu_x\ \vec H = \partial_t\vec D + \vec j,
	\end{equation}
	\begin{equation}\label{eqn:Max3}
		\mathrm{div}_x \vec B = 0,
	\end{equation}
	\begin{equation}\label{eqn:Max4}
		\mathrm{div}_x \vec D = \rho.
	\end{equation}		
\end{subequations}
In equation \eqref{eq:MaxwellSys}, there are four unknowns: The electric field $\vec E$, the magnetic field $\vec H$, the electric flux density $\vec D $  and the magnetic flux density $\vec B$. The variable $\vec j$ is the given electric current density and $\rho$ is the given charge density. 
Additionally to the Maxwell system \eqref{eq:MaxwellSys} we use constitutive relations, also named material laws, relating the electric flux density $\vec D $ with the electric field $\vec E$ and the magnetic field $\vec H$ with the magnetic flux density $\vec B$. In this paper, we will consider the  constitutive relations
\begin{subequations}\label{eq:MaterialLaw}
	\begin{equation}
		\vec D=\varepsilon \vec E,
	\end{equation}
	\begin{equation}
		\vec H=\mu \vec B,
	\end{equation}
\end{subequations}
where $\varepsilon$ is the permittivity and $\mu$ the permeability, which we define in Assumption \ref{A:Assumptions_all} in more detail.	  

How to solve the Maxwell system \eqref{eq:MaxwellSys} has been an active field of study in the last century. One possible solution is going into the frequency domain. For that purpose, we assume that the solutions of system \eqref{eq:MaxwellSys} behave like waves in time. In the literature, we find several introductions to Maxwell's equations. We will not discuss the modeling background in the frequency domain in this paper but simply refer to some of these works such as  \cite{Serdyukov,landauelectrodynamics} or, from a mathematical and numerical point of view \cite{Monk,cessenat1996mathematical}. 
A good introduction to the computational theory of the frequency domain would be \cite{Monk,umashankar1989introduction}. In the last century, the time-harmonic Maxwell's equations were often applied to scattering problems, see e.g.~\cite{jackson1998classical}. Analytical solutions to scattering problems can be found in \cite{poggio1970integral} and theory on inverse scattering and optimal control in scattering by Colton and Kress in 1983, \cite{colton1983integral}. For an introduction to inverse scattering problems, we refer to \cite{coltoninverse}.

Another possibility to solve Maxwell's equations \eqref{eq:MaxwellSys} is the time-stepping method which is introduced in e.g. \cite[Ch.~12.2]{Jin} and references there. However,  the time-stepping approach {exhibits} instabilities. 	A possibility to deal with these complications is to stabilize the system which is created by the time-stepping method. Examples can be found in, e.g., \cite{bai2022second,XIE2021109896,crawford2020unconditionally}. However, since the time-stepping method is a type of finite difference method, it cannot deal with complex geometries.

The possibility to deal with complex geometries as well as going into higher order is the main advantage of space-time finite element methods. In this paper, we will set the theoretical background that is needed for such a method and explore the possibilities and restrictions of the vectorial wave equation that is derived from Maxwell's equations. We will also study elements of second order in time when we apply the theory to a finite element method. The goal of this paper is to derive a space-time variational formulation for the vectorial wave equation, show its unique solvability and analyze its numerical properties. 

First, we will start by deriving the vectorial wave equation from Maxwell's equations. For the modeling, we will assume that all functions in \eqref{eq:MaxwellSys} are smooth and we consider a space-time domain $Q$ in $\R^4$ which is star-like with respect to a ball $B_r$ and Lipschitz in space. The star-like property implies {that} the convex hull of any {$x\in Q$} and $B_r$ is contained in  {$Q$}.  Let us not treat time differently from space, but rewrite the equations such that the partial derivative in time is simply a partial derivative of a space-time derivative
{and therefore rewrite Maxwell's equations into a 4D system of equations. For this purpose, we use the exterior calculus, which allows for a metric-independent generalisation of Stokes' and Gauss' theorem. In the exterior calculus, we consider the exterior derivative $\mathrm{d}$, which extends the concept of the derivative of a function to differential forms of higher degree. This allows us to rewrite Maxwell's equations as a wave equation and implies which function spaces are a natural choice for the test and ansatz spaces of the variational formulation for this wave equation.
	
	Let us therefore formulate Maxwell's equations using the exterior calculus. A detailed introduction to the formalism can be found, for example, in \cite{fortney,lang}. Here we only give the most important considerations for deriving the vector wave equation.  In this context, there are two important differential forms, namely the Faraday 2-form $F$ and the Maxwell 2-form $G$. The Faraday 2-form $F$ contains the electric field $\vec E$ and the magnetic flux density $\vec B$, while the Maxwell 2-form contains the electric flux density $\vec D$ and the magnetic field $\vec H$. The source term $\mathcal{J}$ appears on the right-hand side and includes the current density $\vec j$ and the charge density $\rho$. On the other side, we want to include the material law \eqref{eq:MaterialLaw}. This is done by using a weighted Hodge star operator $\star^{\varepsilon,-\mu^{-1}}$, which maps the Faraday 2-form $F$ to the Maxwell 2-form $G$ using the material law \eqref{eq:MaterialLaw}.
	
	If these definitions are inserted into Maxwell's equations and a metric-independent representation is used, the following equations are obtained
	\begin{align}
		\nonumber
		\mathrm{d}\ F &= 0,\\
		\label{eqn:Maxwell4D}
		\mathrm{d}\ G &= \mathcal{J},\\
		\nonumber
		G &= \star^{\varepsilon,-\mu^{-1}} F.
	\end{align}
	For the derivation of the equations consider \cite{stern} or \cite[p.~135]{post1962formal}. The third equation represents the material law \eqref{eq:MaterialLaw} and includes the weighted Hodge star operator. As an example for $\R^4$ with the Euclidean metric $\varepsilon$ is the weight in the direction $\star(\mathrm{d}x^{01},\mathrm{d}x^{02},\mathrm{d}x^{03})^T$ and $(-\mu^{-1})$ in the direction $\star(\mathrm{d}x^{23},\mathrm{d}x^{31},\mathrm{d}x^{12})^T$. 
	
	Let us take a closer look at the equations in \eqref{eqn:Maxwell4D}.   The second equation, $\mathrm{d}\ G = \mathcal{J}$,  includes the second, \eqref{eqn:Max2}, and fourth equation, \eqref{eqn:Max4}, of the system \eqref{eq:MaxwellSys}. The first equation, $\mathrm{d}\ F = 0$,  includes the first, \eqref{eqn:Max1}, and third equation \eqref{eqn:Max3} of the system \eqref{eq:MaxwellSys}. }
Moreover, the first equation in \eqref{eqn:Maxwell4D} implies that the Faraday form $F$ is closed and since the domain $Q$ is starlike, the Poincaré-Lemma, \cite[Thm.~4.1]{lang}, can be {applied}. The Poincar\'{e}-Lemma tells us that the form $F$ is closed and therefore exact, i.e. there exists a 1-form $\mathcal A$ such that $\mathrm{d}\mathcal A = F.$

If we insert $F = \mathrm{d}\mathcal A$ into the second equation combined with the third equation of (\ref{eqn:Maxwell4D}), we derive the following wave-type equation
\begin{align}
	\label{eqn:PotentialEquation}
	\mathrm{d} \star^{\varepsilon,-\mu^{-1}} \mathrm{d} \mathcal A =  \mathcal{J}.
\end{align}
Additionally, we have the following two relations
\begin{align*}
	\vec E &= -\partial_t \underline{A} +\nabla_x \vec A_0,\\
	\vec B &= \nabla\times\underline{A},
\end{align*}
where $A_0$  is the time component and $\underline{A}:=(A_1,A_2,A_3)^T$ the spatial component of $\mathcal A,$ see e.g.~\cite[p.~389]{fortney}. 
Note that \eqref{eqn:PotentialEquation} will result in the scalar wave equation if we use a differential form of order zero instead of order one for $\mathcal A$, while the equation \eqref{eqn:PotentialEquation} as such will end up in the vectorial wave equation. Hence, in terms of differential forms in 4D, the scalar and the vectorial wave equation are closely related.

Moreover, we know that $\mathcal A$ is not unique. If we take a 0-form $\tilde{\mathcal A}$ and add $\mathrm{d}\tilde{\mathcal A}$ to $\mathcal A$ then $F=\mathrm{d}\mathcal A=\mathrm{d}\mathcal A+	\mathrm{d}\mathrm{d}\tilde{\mathcal A} = \mathrm{d}(\mathcal A+\mathrm{d}\tilde{\mathcal A})$. If we take a look at the corresponding $L^2$ {function spaces}, we see that this translates to adding the space-time gradient of any $H^1_0(Q)$-function to the potential and getting another viable potential. Hence we need a gauge to fix the potential $\mathcal A$. There are many gauges, see e.g.~a paper from 2001 on the history of gauge invariance \cite{jackson2001historical}. 

Depending on the gauge we can derive different equations from \eqref{eqn:PotentialEquation}. To derive the vectorial wave equation, we use the Weyl gauge, also named the temporal gauge, $A_0=0.$ Then we consider \eqref{eqn:PotentialEquation} in euclidean space and arrive at
\begin{subequations}	
	\label{eqn:Betw_Pot_and_VecWav}
	\begin{equation}
		\label{eqn:Betw_Pot_and_VecWav1}
		\partial_t(\varepsilon (\partial_t \underline{A}))+\cu_x\left(\mu^{-1}\cu_x\underline{A}\right) =\underline{j},\\
	\end{equation}
	\begin{equation}
		\label{eqn:Betw_Pot_and_VecWav2}
		\textmd{div}_x(\varepsilon (\partial_t \underline{A}) )=-\rho.
	\end{equation}
\end{subequations}
Note that the equation \eqref{eqn:Betw_Pot_and_VecWav1} is the vectorial wave equation. In addition to these two equations, we assume that the charge is preserved and therefore the continuity equation 
\begin{align}
	\label{eqn:ContinuityEq}
	\partial_t \rho +\mathrm{div}_x\underline{j} =0
\end{align}
holds true. However, this implies assumptions on the regularity of $\vec A$ and $\vec j$ since we are interested in the second-order time derivative of $\vec A$ and the divergence of $\vec j$. This will  be discussed further in the section on finite element solutions. The continuity equation is, however, already included in the combination of both equations of \eqref{eqn:Betw_Pot_and_VecWav}. 
On the other hand, we can rewrite the second equation  \eqref{eqn:Betw_Pot_and_VecWav2} into initial conditions for $\partial_t \vec A$. Indeed, if we assume enough regularity of $\vec A,$ $\rho,$ and $\vec j$, we can take the spatial divergence of the first equation \eqref{eqn:Betw_Pot_and_VecWav1} and use the continuity equation to derive
\begin{align*}
	\partial_t (\textmd{div}_x (\varepsilon \partial_t\vec A(t)) + \rho ) =0.
\end{align*}
Hence $\textmd{div}_x (\varepsilon \partial_t\vec A(t))(t) + \rho(t) =\textmd{div}_x (\varepsilon \partial_t\vec A(t))(0) + \rho(0) $, for all $t\in (0,T]$. Therefore, if $ \partial_t \vec A(0)$ satisfies
\begin{align*}
	\textmd{div}_x (\varepsilon \partial_t\vec A(0))  = -\rho(0)
\end{align*}
then we satisfy  \eqref{eqn:Betw_Pot_and_VecWav2} for all $t> 0$.

\subsection{The vectorial wave equation under consideration of Ohm's law}

Now that we derived the vectorial wave equation, let us take a look at Ohm's law. In the case of a conducting material, the electromagnetic field itself induces currents and in the easiest case this can be modeled by Ohm's law. Hence, we want to include Ohm's law into our equation which is given by 
\[
\vec j = \sigma \vec E+ \vec j_a,
\]
where $\sigma$ is the conductivity and $\vec j_a$ the applied current.  The relationship between $ \vec E$ and the magnetic vector potential $\vec A$ is given by $\vec E = -\partial_t \vec  A$. Therefore, in this paper, we consider the following equation
\begin{subequations}
	\begin{alignat}{4}
		\varepsilon\partial_{tt}\vec A+  \sigma \partial_t \vec A+\cu_x(\mu^{-1}\cu_x \vec A)  &= \vec j_a &&\textmd{ in }Q = (0,T)\times\Omega,\\
		\partial_t \vec A(0,x) &= \vec \psi(x) &&\textmd{ in }\Omega,\\
		\vec A(0,x) &=  \vec \phi(x) &&\textmd{ in }\Omega,\\
		\trt  \vec A &= 0 &&\textmd{ on }\Sigma =\{0,T\}\times\partial\Omega,
	\end{alignat}
	\label{eqn:H2_curl-dgl_drhs}
\end{subequations}
where $\trt $ is the tangential trace of $\vec A$ on $\Omega$.

Before we state the general assumption of this paper on the functions in \eqref{eqn:H2_curl-dgl_drhs}, we have to define a few  function spaces that will be used throughout this paper. First, we define for $Q\subset \R^d$, $d\in \N,$ the space $L^2(Q;\R^d)$ as the usual Lebesgue space for vector-valued functions $v \colon \, Q \to \R^d$ with the inner product $(v,w)_{L^2(\Omega)} := (v,w)_{L^2(\Omega;\R^d)}:= \int_Q( v(x), w(x))_{\R^d} \mathrm dx$  for $v,w \in L^2(Q;\R^d)$ and the induced norm $\norm{\cdot}_{L^2(Q)} := \norm{\cdot}_{L^2(Q; \R^d)} := \sqrt{(\cdot,\cdot)_{L^2(Q)}}$.  Second, we define $L^\infty(\Omega;\R):= L^\infty(\Omega)$ for $\Omega\subset \R^d$, $d\in\N$, as the space of measurable functions bounded almost everywhere and equipped with the usual norm $\norm{\cdot}_{L^\infty(\Omega)}$. Third, we define the space $L^\infty(\Omega;\R^{d\times d})$, $d\in\N$, as the space of matrix-valued measurable functions bounded almost everywhere with the norm
\[	\|w\|_{L^\infty(\Omega)} := \|w\|_{L^\infty(\Omega;\R^{d\times d})} :=\esssup_{x\in \Omega}  \sup_{0 \ne \xi\in \R^d} \frac{\xi^\top  {w}(x) \xi}{\xi^\top\xi}.\]
Finally, we define space-time spaces $L^2(0,T;X)$ with the inner product $(v,w)_{L^2(0,T;X)} \coloneqq  \int_0^T ( v(t), w(t))_{X} \mathrm dt$ and $L^1(0,T;L^2(\Omega;\R^d))$ with the norm  $\norm{\vec v}_{2,1,Q}\coloneqq \norm{\vec v}_{L^1(0,T;L^2(\Omega;\R^d))}\coloneqq \int_0^T\norm{v}_{L^2(\Omega)} \mathrm dt$ in the same way as above. Additionally, we can define the space $H^1(0,T;X)$  as the Hilbert space $H^1(0,T)$ over the Hilbert space $X$, see \cite{zeidler} for more details.

With these definitions we define the tangential trace operator $\trt$ for $d=3$ as the continuous mapping $\trt \colon \, H(\cu;\Omega) \to H^{-1/2}(\partial \Omega;\R^3)$. The tangential trace operator is the unique extension of the vector-valued function $\trt \vec v = \vec v_{|\partial \Omega} \times \vec n_x$ for $\vec v \in H^1(\Omega)^3$ that is defined by the Green's identity for the $\cu$ operator. For $d=2$ we define the tangential trace operator $\trt \colon \, H(\cu;\Omega) \to H^{-1/2}(\partial \Omega)$ as the unique extension of the scalar function $\trt \vec v = \vec v_{|\partial \Omega} \cdot \vec \tau_x$ for $\vec v \in H^1(\Omega)^2$, where $\vec \tau_x $ is the unit tangent vector, i.e., $\vec \tau_x \cdot \vec n_x = 0$.  We define the usual trace operator  $\gamma: H^1(\Omega,\R^d)\to H^{1/2}(\partial \Omega;\R^d)$  as  $\gamma \vec v := (v)_{|\partial \Omega}$  for $d \in \N $ and  $H^{-1/2}(\partial \Omega)$ is the dual space of the  Sobolev space $H^{1/2}(\partial \Omega)$. For more details on the tangential trace operator $\trt$, consider \cite{Assous2018,ErnGuermond2020I,Monk}. 

For detailed definitions of the spaces $H_0(\cu;\Omega) := \{\vec v\in H(\cu;\Omega)\,|\, \trt \vec v = 0\}$ and $H(\mathrm{div};\Omega)$, we reference \cite{Monk}.  However, we quickly note how the curl operator behaves differently in two and three dimensions. For $d=2$, we set the curl of a sufficiently smooth vector-valued function $\vec v \colon \, \Omega \to \R^2$ with $\vec v = (v_1,v_2)^\top$ as the scalar-valued curl operator $\cu_x \vec v = \partial_{x_1} v_2 - \partial_{x_2} v_1$. Sometimes this curl operator is called rot. Additionally, for a sufficiently smooth scalar function $w \colon \, \Omega \to \R$, we define the vector-valued curl operator $\cu_x w=(\partial_{x_2} w, - \partial_{x_1} w)^\top.$ Note that the vector-valued curl operator is the adjoint operator of the scalar-valued curl operator for $d=2$.\\	
In the case of $d=3$, the curl of a sufficiently smooth vector-valued function $\vec v \colon \, \Omega \to \R^3$ with $\vec v= (v_1, v_2, v_3)^\top$ is given by the vector-valued function
\[\cu_x \, \vec v =\left( \partial_{x_2} v_3 - \partial_{x_3} v_2, \partial_{x_3} v_1 - \partial_{x_1} v_3, \partial_{x_1} v_2 - \partial_{x_2} v_1 \right)^\top. \]

With these definitions, we can write the following assumptions that hold throughout the paper.
\begin{assumption} \label{A:Assumptions_all}
	Let $d=2,3$ and the spatial domain $\Omega \subset \R^d$ and  $\mathrm{supp}(\sigma)\subset \R^d$, $\sigma: \Omega\to \R^{d\times d}$, be given such that
	\begin{itemize}
		\item $\Omega$ and $\Omega_\sigma := \mathrm{supp}(\sigma)\subset\Omega$ are Lipschitz domains,
		\item and $Q = (0,T) \times \Omega$ is a star-like domain with respect to a ball $B$.
	\end{itemize}
	Further, let $\sigma$, $\vec j$, $\varepsilon$ and $\mu$ be given functions, which fulfill:
	\begin{itemize}
		\item The conductivity $\sigma \in L^\infty(\mathrm{supp}(\sigma); \R^{d \times d})$, $\mathrm{supp}(\sigma)\subset \Omega,$ is uniformly positive definite, i.e.
		\begin{equation*}
			\sigma_{\min} := \essinf_{x\in \mathrm{supp}(\sigma)} \inf_{0 \ne \xi\in \R^d} \frac{\xi^\top \sigma(x) \xi}{\xi^\top\xi} > 0.
		\end{equation*}		
		\item The applied current density $\vec j_a \colon \, Q \to \R^d$ satisfies $\vec j_a \in L^1(0,T;L^2(\Omega;\R^d))$.
		\item The permittivity $\varepsilon \colon \, \Omega \to \R^{d \times d}$ is symmetric, bounded, i.e. $\varepsilon \in L^\infty(\Omega; \R^{d \times d})$, and uniformly positive definite, i.e., 
		\begin{equation*}
			\varepsilon_{\min} := \essinf_{x\in \Omega} \inf_{0 \ne \xi\in \R^d} \frac{\xi^\top \varepsilon(x) \xi}{\xi^\top\xi} > 0
		\end{equation*}
		and $\varepsilon^{\max}:= \norm{\varepsilon}_{L^\infty(\Omega)}$.
		\item For $d=2$, the permeability $\mu \colon \, \Omega \to \R$ satisfies $\mu \in L^\infty(\Omega;\R)$, $\mu^{\max}:= \norm{\mu}_{L^\infty(\Omega)}$, and
		\begin{equation*}
			\mu_{\min} := \essinf_{x\in \Omega} \mu(x)  > 0.
		\end{equation*}		
		For $d=3$, the permeability $\mu \colon \, \Omega \to \R^{3 \times 3}$  is symmetric, bounded, i.e. $\mu \in L^\infty(\Omega; \R^{3 \times 3})$,  $\mu^{\max}:= \norm{\mu}_{L^\infty(\Omega)}$, and uniformly positive definite, i.e., 
		\begin{equation*}
			\mu_{\min} := \essinf_{x\in \Omega} \inf_{0 \ne \xi\in \R^3} \frac{\xi^\top \mu(x) \xi}{\xi^\top\xi} > 0.
		\end{equation*}
	\end{itemize}
	For the initial data $\vec \phi$ and $\vec \psi$ we assume that:
	\begin{itemize}
		\item The initial data $\vec \phi:\Omega \to \R^d$ satisfies $\phi \in H_0(\cu_x;\Omega)$.
		\item The initial data $ \vec \psi \in H(\mathrm{div}_x;\Omega)$ satisfies $\mathrm{div}_x (\varepsilon \vec \psi)=-\rho(0)$ in $\Omega$. 
	\end{itemize}
\end{assumption}
Now that the assumptions are stated and the vectorial wave equation is introduced, let us take a look at the structure of the rest of the paper. 
In the remainder of this section, we will introduce the  function spaces that are needed to formulate the space-time variational formulation of the vectorial wave equation. We will discuss that they are indeed Hilbert spaces and go over possible basis representations. At the end of this section, we will derive the variational formulation. The second section of this paper is dedicated to the proof of unique solvability for the variational formulation and norm estimates of the solution. In this proof, we will use a Galerkin method and use the previously developed basis representation. In the third section of this paper, we will discuss the space-time finite element spaces and discretization. We will derive a CFL condition and take a look at examples that   illustrate this CFL condition. In the end, we will sum up the conclusions and give an outlook.

\subsection{The space-time  Sobolev spaces}
To derive a variational formulation, we need to consider the appropriate  function space for  the potential $\mathcal A$ satisfying \eqref{eqn:PotentialEquation}. In this section, we will derive them and show properties that will be needed in this paper.\\
To derive the  function spaces let us assume for a moment that $d=3$, then $Q\subset \R^4$. From the derivation of the magnetic vector potential $\mathcal A$ we know that $\mathcal A$ is a $1$-form in $\R^4$. For $1$-forms the corresponding  function space in the $L^2$-Hilbert complex in $\R^4$  is $ H(\mathrm{Curl},Q;\R^4)$, see e.g.~\cite{4DDeRham}. The space $ H(\mathrm{Curl},Q;\R^4)$ is defined by
\begin{align*}
	H(\mathrm{Curl},Q;\R^4) &\coloneqq \lbrace \ \vec u\in L^2(Q;\R^4)\ |\ \mathrm{Curl}\ \vec u \in L^2(\Omega;\mathbb{K}) \ \rbrace,\\
	[\mathrm{Curl}\ \vec u]_{ij} &\coloneqq \sum_{k,l=1}^4 \varepsilon_{ijkl}\partial_ku_l,
\end{align*}
see e.g.~\cite{4DDeRham},  where $\mathbb{K}$ is the vector space of $4\times 4$ skew symmetric matrices. Note, that the operator $\mathrm{Curl}$ is an operator in both space and time. Additionally, we use the Weyl gauge, the temporal gauge which implies $A_0=0$, to make the potential $\mathcal A$ unique in the derivation of the vectorial wave equation. Hence, we are interested in a function of the type $(0,A_1,A_2,A_3)^T$ that is in $H(\mathrm{Curl},Q;\R^4)$. Rewriting the operator $\mathrm{Curl}$ for $A_0=0$ and $A = (A_0,\underline{A})^T\in H(\mathrm{Curl},Q;\R^4)$ leads to the conditions
\begin{equation}
	\label{def:FuncCond}
	\underline{A}\in L^2(Q,\R^3),\qquad
	\cu_x\underline{A}\in L^2(Q,\R^3),\quad \text{and}\quad
	\partial_t\underline{A}\in L^2(Q,\R^3).
\end{equation}

Before we define the appropriate function space, we take a quick look at the coefficients in the vectorial wave equation \eqref{eqn:H2_curl-dgl_drhs}. When we will formulate a variational formulation for the equation \eqref{eqn:H2_curl-dgl_drhs}, we will use the weighted $L^2(\Omega;\R^d)$-inner products
\begin{align*}
	(\vec u, \vec v)_{L^2_\varepsilon(\Omega)} &:= \left( \varepsilon \vec u, \vec v \right)_{L^2(\Omega)} ,\\
	(\vec u, \vec v)_{L^2_\mu(\Omega)} &:=  \left( \mu^{-1} \vec u, \vec v \right)_{L^2(\Omega)} ,
\end{align*}
for $\vec v, \, \vec w \in L^2(\Omega;\R^d)$ and the $H_0(\cu;\Omega)$-inner product
\begin{equation*}
	(\vec u, \vec v)_{H_{0,\varepsilon,\mu}(\cu;\Omega)} := (\vec u, \vec v)_{L^2_\varepsilon(\Omega)} +  (\cu_x\vec v,\cu_x \vec w)_{L^2_\mu(\Omega)},
\end{equation*}
for $ \vec v, \, \vec w \in H_0(\cu;\Omega).$
Note, that due to the Assumption~\ref{A:Assumptions_all} the norms $\norm{\cdot}_{L^2_\varepsilon(\Omega)}$ and $\norm{\cdot}_{H_{0,\varepsilon,\mu}(\cu;\Omega)}$, induced by $(\cdot, \cdot)_{L^2_\varepsilon(\Omega)}$ and $(\cdot, \cdot)_{H_{0,\varepsilon,\mu}(\cu;\Omega)}$, are equivalent to the standard norms $\norm{\cdot}_{L^2(\Omega)}$ and $\norm{\cdot}_{H_0(\cu;\Omega)}$, respectively.

Now let us define the function spaces which we will be using from this point on.  Combining the conditions in \eqref{def:FuncCond} results in  the following spaces 
\begin{align}
	\nonumber
	H^{\cu;1}_{0;0,}(Q) &:=  \{\ \vec u\in L^2(Q,\R^d) \ | \  \partial_t \vec u\in L^2_{\varepsilon}(Q),\ \cu_x \vec u\in L^2_{\mu}(Q),\ \vec u(0,x)= 0 \textmd{ for } x\in \Omega,\\
	\label{def:spaceHC10,}
	&\hspace*{1cm} \gamma_t \vec u = 0 \textmd{ on }\Sigma = (0,T)\times\partial \Omega \   \}\\
	\nonumber
	&= L^2(0,T;H_0(\cu;\Omega))\cap H^1_{0,}(0,T;L^2(Q,\R^d)),\\
	\nonumber
	H^{\cu;1}_{0;,0}(Q) &:=  \{\ \vec u\in L^2(Q,\R^d) \ | \  \partial_t\vec  u\in L^2_{\varepsilon}(Q),\ \cu_x \vec u\in L^2_{\mu}(Q),\ \vec u(T,x)= 0 \textmd{ for } x\in \Omega,\\
	\label{def:spaceHC1,0}
	&\hspace*{1cm} \gamma_t \vec u = 0 \textmd{ on }\Sigma \   \}\\
	\nonumber
	&= L^2(0,T;H_0(\cu;\Omega))\cap H^1_{,0}(0,T;L^2(Q,\R^d)).
\end{align}
The subscript '$0,$' and '$,0$' stands for zero initial and zero end conditions. In Lemma~\ref{lem:space_time_HR} we will see that they are well defined.

	Since $ H^{\cu;1}_{0;}(Q) $ is the kernel of a bounded space-time trace it is quite natural to assume that  $ H^{\cu;1}_{0;}(Q) $ is a Hilbert space. We will quickly discuss whether this is indeed true and we state in which concept $C^\infty(Q)$-functions are dense in our space. We end up looking at the  seminorm in this setting which is a good tool for numerical analysis.\\
So let us start with showing that  $H^{\cu;1}_{0;0,}(Q)$ and $H^{\cu;1}_{0;,0}(Q)$ are Hilbert spaces.

\begin{lem}\label{lem:C0inf_dense}
	Let $d=2,3$. The space $	H^1(0,T;L^2(\Omega;\R^d)) $ is isometric to the Hilbert tensor product $H^1(0,T)\hat{\tens}L^2(\Omega;\R^d)$ and the space $L^2(0,T;H(\cu;\Omega))$ is isometric to the Hilbert tensor product $ L^2(0,T)\hat{\tens}H(\cu;\Omega)$.\\
	Addtionally the space $C^\infty(0,T){\tens} [C^\infty(\overline{\Omega})]^d$ is dense in
	\begin{align*}
		H^1(0,T)\hat{\tens}L^2(\Omega;\R^d)\ \mathrm{ and }\	L^2(0,T)\hat{\tens}H(\cu;\Omega)
	\end{align*}
	and $C^\infty([0,T]){\tens} [C_0^\infty(\Omega)]^d $ is dense in $	L^2(0,T)\hat{\tens}H_0(\cu;\Omega).$
\end{lem}
See Aubin \cite{aubin}, Thm.~12.7.1 and Thm.~12.6.1. . The second result follows with \cite[ Ch 1.6.]{weidmann}.  Here, $\tens$  denotes the tensor product and $\hat{\tens}$ the tensor product of Hilbert spaces, where the resulting space is a Hilbert space itself. Hence we get that  {$C^\infty(0,T;C^\infty(\Omega)^d)$}  is  dense in both spaces $H^1(0,T;L^2(\Omega;\R^d))$ and $L^2(0,T;H(\cu;\Omega))$ and therefore is dense in the intersections $H^{\cu;1}_{0;0,}(Q)$ and $H^{\cu;1}_{0;,0}(Q)$.\\
\begin{lem}\label{lem:space_time_HR}
	The spaces $H^{\cu;1}_{0;0,}(Q)$ and $H^{\cu;1}_{0;,0}(Q)$ are Hilbert spaces equipped with the inner product 
	\begin{align}
		\label{eqn:scalar_prod}
		(\vec u, \vec v)_{H^{\cu;1}(Q)} := (\vec u,\vec v)_{L^2(Q)} +(\partial_t \vec u,\partial_t \vec v)_{L^2_{\varepsilon}(Q)}+(\cu_x \vec u,\cu_x \vec v)_{L^2_{\mu}(Q)}
	\end{align}
	for all $ \vec u,\vec v \in H^{\cu;1}_{0;}(Q)$. Additionally the initial and end conditions are well defined in $H^{\cu;1}_{0;}(Q)$
\end{lem}
\begin{proof}
	Since the embedding $H^1(0,T;L^2(\Omega;\R^d)) \subset C([0,T], L^2(\Omega;\R^d))$ is continuous, see \cite[Prop 23.23]{zeidler},  we get that $
	H^1(0,T;L^2(\Omega;\R^d))\subset L^2(Q;\R^d)$	is continuously embedded, see \cite[Prop 23.2]{zeidler}, and since $H_0(\cu;\Omega) \subset L^2(Q;\R^d)$ that $
	L^2(0,T;H_0(\cu;\Omega))\subset L^2(Q;\R^d)$ is continuously embedded, see \cite[Prop 23.2]{zeidler}. We know that  $L^2(Q;\R^d)$ is continuously embedded in the Hausdorff space $\mathcal{M}_0$, see \cite[Thm.~I.1.4]{bennett},  where space $\mathcal{M}_0$ consists of all Lebesgue measurable functions defined on $Q$ that are finite a.e..  Therefore the pair $(H^1(0,T;L^2(\Omega;\R^d)), L^2(0,T;H_0(\cu;\Omega)))$ is a compatible couple,  while a compatible pair consists of two linear subspaces of a larger vector space. In that case the intersection of both spaces is again a linear subspace of the larger space and a Banach space with the maximum of both norms, i.e. $	H^1(0,T;L^2(\Omega;\R^d))\cap L^2(0,T;H_0(\cu;\Omega))$	is a Banach space with the norm
	\begin{align*}
		\| \vec u\|_{H^1(0,T;L^2(\Omega;\R^d))\cap L^2(0,T;H_0(\cu;\Omega))} = \max \{	\| \vec  u\|_{H^1(0,T;L^2(\Omega;\R^d))},	\|\vec  u\|_{ L^2(0,T;H_0(\cu;\Omega))} \},
	\end{align*}
	see \cite[Thm.~III.1.3]{bennett}. Now, \eqref{eqn:scalar_prod} defines a norm that is induced by the inner product and it is equivalent to the above norm. Hence $H^{\cu;1}_{0;}(Q)$ is a Hilbert space.\\
	Because the embedding $H^1(0,T;L^2(\Omega;\R^d)) \subset C([0,T], L^2(\Omega;\R^d))$ is continuous we get well defined and continuous traces to the boundaries $\{t=0\}\times \Omega$ and $\{t=T\}\times \Omega$ for $H^{\cu;1}_{0;}(Q)$. Since the spaces $H^{\cu;1}_{0;0,}(Q)$ and $H^{\cu;1}_{0;,0}(Q)$ are the kernels of these traces, we derive that they are closed subsets of the Hilbert space $	H^{\cu;1}_{0;}(Q)$.
\end{proof}

\begin{lem}\label{lem:friedrich}
	For the spaces $H^{\cu;1}_{0;0,}(Q)$ and $H^{\cu;1}_{0;,0}(Q)$ there  exists  $c_f>0$ such that
	\begin{align*}
		\|\vec u\|_{L^2(Q)} \leq c_f \|\partial_t \vec u\|_{L^2(Q)} 
	\end{align*}
	for all $\vec u\in H^{\cu;1}_{0;0,}(Q) $ or $\vec u\in H^{\cu;1}_{0;,0}(Q), $ $d=2,3$. Therefore the  seminorm
	\begin{align*}
		|u|^2_{H^{\cu;1}(Q)} := \|\partial_t\vec u\|^2_{L^2_\varepsilon(Q)}+\|\cu_x \vec u\|^2_{L^2_\mu(Q)},
	\end{align*}
	is an equivalent norm in $H^{\cu;1}_{0;0,}(Q)$ and $H^{\cu;1}_{0;,0}(Q)$.
\end{lem}
This can be  proved simply by using the Poincaré-inequality in $H^1(0,T)$ and the structure of the norm $\norm{.}_{H^1(0,T;X)}$, see e.g.~\cite[Sec.~4.1]{HauserZank2023} for more details.

\subsubsection{Basis representations}\label{sec:BasisRep}
To derive a basis representation in $H^{\cu;1}_{0;}(Q)$ we first have to state how we decompose  $ H_0(\cu;\Omega)$. The basis representation of  $H^{\cu;1}_{0;}(Q)$  will be used in the proof of the main theorem on uniqueness and solvability. To understand the proof better, we will quickly derive the decomposition.
 Let us define the subspace
	\begin{equation*}
		H(\di \epsilon 0;\Omega) := \left\{ \vec f \in L^2(\Omega;\R^d) : \, \di_x (\epsilon \vec f) = 0 \right\} \subset L^2(\Omega;\R^d)
	\end{equation*}
	endowed with the weighted inner product $(\cdot, \cdot)_{L^2_\epsilon(\Omega)}$, whereas the subspace
	\begin{equation*}
		X_{0,\epsilon}(\Omega) := H_0(\cu;\Omega) \cap H(\di \epsilon 0;\Omega) \subset H_0(\cu;\Omega)
	\end{equation*}
	is equipped with the inner product $(\cdot, \cdot)_{H_{0,\epsilon,\mu}(\cu;\Omega)}$. With this notation, we recall a crucial decomposition result, the Helmholtz-Weyl decomposition \eqref{def:Helmholtz} of $H_0(\cu;\Omega)$.
	\begin{lem} \label{Lem:Helmholtz}
		Let Assumption~\ref{A:Assumptions_all} be satisfied. Then, the the orthogonal decomposition
		\begin{equation} \label{def:Helmholtz}
			H_0(\cu;\Omega) = \nabla_x H^1_0(\Omega) \oplus X_{0,\epsilon}(\Omega)
		\end{equation}
		is true, where the orthogonality holds true with respect to $(\cdot, \cdot)_{H_{0,\epsilon,\mu}(\cu;\Omega)}$, $(\cdot, \cdot)_{H_{0,\mu}(\cu;\Omega)}$ and $(\cdot, \cdot)_{L^2_\epsilon(\Omega)}$.
	\end{lem}
	\begin{proof}
		For $d=3$, the result is stated in \cite[Proposition~7.4.3]{Assous2018}. Additionally, for $d \in \{2,3\}$, the decompositions follow from properties of the corresponding de Rham complexes, see e.g. \cite[Section~2.3]{ArnoldFalkWinther}, \cite[Lemma~2.7, Lemma~3.6, Theorem~5.5]{PaulHilbertComplexes}.
	\end{proof}

\begin{lem}\label{lem:FS-H_0curl}
	The space $H_0(\cu;\Omega)$ has a fundamental system $\{\vec \varphi_k\}_{k\in\Z}$ which is orthonormal in the  $L^2_{\varepsilon}(\Omega)$-product. Additionally  $\{\vec \varphi_k\}_{k\in\Z}$  is constructed in such a way that for every $k\in \N_0$ there exists a  $\lambda_k> 0$ such that
	\begin{align*}
		(\cu_x \vec \varphi_k,\cu_x \vec v)_{L^2_\mu(\Omega)} =  \lambda_k (\vec \varphi_k, \vec v)_{L^2_\varepsilon(\Omega)}
	\end{align*}
	for all $ \vec v\in X_{0,\varepsilon}(\Omega)$ and for $k\in \Z\backslash\N_0$ there exists a $\lambda_k>0$ and $\phi_k\in H^1_0(\Omega)$ such that $\vec \varphi_k = \nabla_x \phi_k$ and 
	\begin{align*}
		(\nabla_x\phi_k, \nabla_x v)_{L^2_\varepsilon(\Omega)} = \lambda_{k} (\phi_k, v)_{L^2_\varepsilon(\Omega)}.
	\end{align*}
	for all	$ v\in H^1_0(\Omega)$.
\end{lem}
\begin{proof}
	For $\nabla H^1_0(\Omega) $ we get an orthonormal basis from the Laplace eigenvalue problem: \\
	Find  $(\lambda_k,\phi_k) \in (\R, H^1_0(\Omega))$, $ k \in \Z\backslash\N_0$, such that for all $ v\in H^1_0(\Omega)$
	\begin{equation*}
		(\nabla_x \phi_k, \nabla_x v)_{L^2_\varepsilon(\Omega)} = \lambda_k (\phi_k, v)_{L^2_{\varepsilon}(\Omega)}\quad \textmd{and} \quad \norm{\nabla_x\phi_k}_{L^2_\varepsilon(\Omega)}= 1.
	\end{equation*}
	The solution to the eigenvalue problem is a non-decreasing sequence of related eigenvalues $\lambda_k > 0$, satisfying $\lambda_k \to \infty$ as $k \to -\infty,$ see \cite[Section~4 in Chapter~4]{Ladyzhenskaya1985}. 
	
	Next, we investigate the eigenvalue problem:\\
	Find  $(\lambda_k,\vec \varphi_k) \in (\R,X_{0,\varepsilon}(\Omega))$, $ k \in \N_0$, such that 	for all $ \vec v \in X_{0,\varepsilon}(\Omega)$
	\begin{equation} \label{VF:Eigen_curl}
		(\vec  \varphi_k, \vec v)_{H_{0,\varepsilon,\mu}(\cu;\Omega)} = (1+\lambda_k) (\vec \varphi_k, \vec v)_{L^2_\varepsilon(\Omega)} \quad \textmd{and} \quad\norm{\vec \varphi_k}_{L^2_\varepsilon(\Omega)}= 1.
	\end{equation}
	The set of eigenfunctions $\{ \vec \varphi_k \in X_{0,\varepsilon}(\Omega) : \, k \in \N_0 \}$ form an orthonormal basis of $ H(\di \varepsilon 0;\Omega)$ with respect to $(\cdot, \cdot)_{L^2_\varepsilon(\Omega)}$. Additionally the nondecreasing sequence of related eigenvalues $ (1+\lambda_k)$, satisfying $\lambda_k \to \infty$ as $k \to \infty,$ see \cite[Theorem~8.2.4]{Assous2018}. Note that $\lambda_k>0$, $k\in \N_0$. This can be shown by estimating
	\begin{equation*}
		0 < c_P(\vec \varphi_k, \vec \varphi_k)_{L^2_\varepsilon(\Omega)} \leq (\cu_x\vec \varphi_k,\cu_x \vec \varphi_k)_{L^2_\mu(\Omega)} =  (\vec \varphi_k, \vec \varphi_k)_{H_{0,\varepsilon,\mu}(\cu;\Omega)} - (\vec \varphi_k, \vec \varphi_k)_{L^2_\varepsilon(\Omega)}= \lambda_k
	\end{equation*}
	using the Poincar\'{e}-Steklov inequality, see e.g. \cite[Lem.~44.4]{ErnGuermond2020II}, and  the  variational formulation~\eqref{VF:Eigen_curl} for $\vec v = \vec \varphi_k$ to get the desired result.
	
	Moreover, the set $\{( 1+\lambda_k)^{-1/2} \vec \varphi_k \in X_{0,\varepsilon}(\Omega): \, k\in \N_0 \}$ is an orthonormal basis of $X_{0,\varepsilon}(\Omega)$ with respect to $(\cdot, \cdot)_{H_{0,\varepsilon,\mu}(\cu;\Omega)}$ by construction, see \eqref{VF:Eigen_curl}, and since $X_{0,\varepsilon}(\Omega)\subset  H(\di \varepsilon 0;\Omega)$. Additionally, we see that the set $\{ \vec e_j \in X_{0,\varepsilon}(\Omega) : \, j \in \N_0 \}$ is also orthogonal with respect to $(\cu_x \cdot,\cu_x  \cdot)_{L^2_\mu(\Omega)}$ since it is an equivalent norm in $X_{0,\varepsilon}(\Omega)$ because of the Poincar\'{e}-Steklov inequality, see  \cite[Lem.~44.4]{ErnGuermond2020II}.\\
	Then, by using  Lem.~\ref{Lem:Helmholtz} we arrive at the desired orthonormal basis of $ H_0(\cu;\Omega)$ with the set $\{ \nabla \phi_k\}_{ k \in \Z\backslash\N_0 }\cup\{( 1+\lambda_k)^{-1/2} \vec \varphi_k\}_{ k \in\N_0 } $, which is orthogonal with respect to $(\cdot, \cdot)_{H_{0,\varepsilon,\mu}(\cu;\Omega)}$.	
\end{proof}
Now that we know the fundamental system of $H_0(\cu;\Omega)$ we can write $\vec w \in H_0(\cu;\Omega)$ as
\begin{equation*}
	\vec w(x) = \sum_{k=-\infty}^\infty w_k \vec \varphi_k(x), \quad x \in \Omega,
\end{equation*}
with the coefficients $w_k = (\vec w, \vec \varphi_k)_{L^2_\varepsilon(\Omega)}, \, k \in \Z.$ This basis representation converges in $H_0(\cu;\Omega).$  Then the seminorm $\abs{\cdot}_{H_{0,\mu}(\cu;\Omega)}$ and the norm $\norm{\cdot}_{H_{0,\varepsilon,\mu}(\cu;\Omega)}$ admit the representations
\begin{equation} \label{VF:H0mucurl}
	\abs{\vec w}_{H_{0,\mu}(\cu;\Omega)}^2 = \sum_{k=0}^\infty \underbrace{\lambda_k}_{> 0} \abs{w_k}^2, \quad  \norm{\vec w}_{H_{0,\varepsilon,\mu}(\cu;\Omega)}^2 = \sum_{k=0}^\infty (1+\lambda_k) \abs{w_k}^2 + \sum_{k=1}^\infty \abs{w_{-k}}^2.
\end{equation}
Let $\vec v\in H^{\cu;1}_{0;}(Q)$. Then we learn from \cite[Prop.~23.23]{zeidler} that $\vec v$ coincides on $[0,T]$ with a continuous mapping $ \vec v:[0,T]\to H_0(\cu;\Omega)$  up to a subset of measure zero. Hence we can write for $t\in[0,T]$
\begin{align}\label{eq:BasisRepresentation}
	\vec v(t) =  \sum_{k=-\infty}^\infty v_k(t) \vec \varphi_k, 
\end{align}
for some coefficient functions $v_k:[0,T]\to \R$.

\subsection{The variational formulation}
To derive a suitable variational formulation for \eqref{eqn:H2_curl-dgl_drhs} we multiply the partial differential equation with a test function. By using partial integration both in time  and space we end up with the following variational formulation:\\
Find $ \vec A\in H^{\cu;1}_{0;}(Q) $ with $\vec A(0,.)=\vec \phi$, such that
\begin{align}
	\label{vf:H1VarForm}
	-\left(\varepsilon\partial_t \vec A,\partial_t \vec v\right)_{L^2(Q)} + \left( \sigma\partial_t \vec A , \vec v\right)_{L^2(Q)}+&\left(\mu^{-1}	\cu_x  \vec A, \cu_x \vec v\right)_{L^2(Q)} \\
	\nonumber
	&= \left(\vec  j_a, \vec v\right)_{L^2(Q)} -\left(\varepsilon	\vec \psi, \vec v(0,\cdot)\right)_{L^2(\Omega)} 
\end{align}
for all $\vec v\in H^{\cu;1}_{0;.,0}(Q) $.\\
Note, that the initial condition $\partial_t \vec A(0)=\vec \psi$ is incorporated into the variational formulation in a weak sense while the other conditions $\vec A(0)=\vec \phi$ and $\gamma_t \vec A=0$ are in the ansatz spaces and therefore are satisfied in a strong sense. In \eqref{vf:H1VarForm} we see the main problem of the equation, the different signs in front of the first term and the spatial differential operator. Hence, the bilinear form \eqref{def:BilinearForm} defined by the left-hand side of the variational formulation \eqref{vf:H1VarForm} is not equivalent to any inner product. However, if $\sigma$ is large it acts as a stabilization to the equation.  We  see this phenomenon in the numerical analysis as well.

In this paper, we will take a look at the bilinear form
\begin{equation}\label{def:BilinearForm}
	a_Q(\vec A,\vec v) \coloneqq 	-(\varepsilon\partial_{t}\vec A, \partial_t \vec v)_{L^2(Q)}+(\sigma \partial_t \vec A,\vec v)_{L^2(Q)}  +(\mu^{-1}\cu_x  \vec A, \cu_x \vec v)_{L^2(Q)},
\end{equation}
for  $\vec A\in H^{\cu;1}_{0;}(Q)$ and $\vec\phi\in  H^{\cu;1}_{0;.,0}(Q)$. Additionally, we write the right-hand side of \eqref{vf:H1VarForm} as the linear form
\begin{equation}\label{def:LinearForm}
	F(\vec v) := \left(\vec  j_a, \vec v\right)_{L^2(Q)} -\left(\varepsilon	\vec \psi, \vec v(0,\cdot)\right)_{L^2(\Omega)}.
\end{equation}

\section{Existence and uniqueness}
Let us now state the main existence and uniqueness result for the variational formulation \eqref{vf:H1VarForm}.
\begin{thm}\label{thm:ExistenceAndUniqueness}
	Let the Assumption~\ref{A:Assumptions_all} hold true. Then there exists a unique solution of the variational formulation:\\
	Find $ \vec A\in H^{\cu;1}_{0;}(Q) $ with $\vec A(0,.)=\vec \phi$, such that
	\begin{align*}
		-\left(\varepsilon\partial_t  \vec A,\partial_t \vec v\right)_{L^2(Q)} + \left( \sigma\partial_t \vec A , \vec v\right)_{L^2(Q)}+\left(\mu^{-1}	\cu_x \vec A, \cu_x  \vec v\right)_{L^2(Q)}
		= \left( \vec j_a, \vec v\right)_{L^2(Q)} -\left(\varepsilon\vec \psi, \vec v(0,\cdot)\right)_{L^2(\Omega)} 
	\end{align*}
	for all $\vec v\in H^{\cu;1}_{0;.,0}(Q) $. \\
	If additionally $\vec j_a\in L^2(Q;\R^d)$ then there exist positive constants  $c_\phi$, $c_\phi^c$, $c_\psi$,  and $c_f$  such that
	\begin{align*}
		|\vec A|^2_{H^{\cu;1}(Q)} \leq c_\phi\|\vec \phi\|^2_{L^2_{\varepsilon}({\Omega})} +c_\phi^cT \|\cu_x\vec \phi\|^2_{L^2_\mu(\Omega)} + c_\psi T\| \vec \psi\|^2_{L^2_{\varepsilon}({\Omega})}  + c_f \max\{T,T^2\} \|\vec j_a\|^2_{L^2_{\varepsilon}({Q})}.
	\end{align*}
\end{thm} 
To prove this theorem we will use a Galerkin method and split the proof into three different steps. First, we prove the existence in Proposition~\ref{prop:Existence_H1_drhs}. Then, we show the uniqueness in Proposition~\ref{prop:Uniqueness_H1_drhs}. In the end, we derive the norm estimates in Proposition~\ref{prop:inequalDepRhs} and discuss the dependencies of the coefficients $c_\phi$, $c_\phi^c$, $c_\psi$,  and $c_f$.
\subsection{Existence}
We will start with proving the existence of the variational formulation \eqref{vf:H1VarForm}. For that purpose, we have to state two small results that we will be using in the existence proof in Prop.~\ref{prop:Existence_H1_drhs}.
\begin{lem}\label{lem:H2_ungl_not_var} 
	Let the Assumption~\ref{A:Assumptions_all} hold true and $\vec u\in  H^2(0,T;L^2(\Omega;\R^d))\cap L^2(0,T;H_0(\cu;\Omega))$, $d=2,3$, with	$\vec u|_{t=0} =  \vec \phi$ and $\partial_t \vec  u|_{t=0} =\vec \psi$ be the solution of
	 \begin{align}\label{ieqn:sec_order}
			\left(\varepsilon\partial_{tt}\vec u(t),\partial_t \vec u(t)\right)_{L^2(\Omega)}+\left(\mu^{-1}\cu_x \vec u(t), \cu_x\partial_t \vec u(t)\right)_{L^2(\Omega)}& \leq  \left( \vec j(t), \partial_t\vec u(t)\right)_{L^2(\Omega)},
		\end{align}	 
		for each $t\in(0,T)$ and $\vec j\in L^{1}(0,T;L^2(\Omega;\R^d))$. Then 
	\begin{align*}
		z^{1/2}(t) := \left(\int\limits_{\Omega}\left(\abs{\vec u}^2+\abs{\partial_t \vec u}^2 +\abs{\cu_x \vec u}^2\right)(t,x)\,\mathrm dx\right)^{1/2} \leq c_2(T) z^{1/2}(0)+c_3(T) \|\vec j\|_{2,1,Q},
	\end{align*}
	holds true for all  $t\in [0,T]$. The constants $c_2$ and $c_3$ depend on  $T, \varepsilon_{\min},(\mu^{\max})^{-1},\varepsilon^{\max}$ and $(\mu_{\min})^{-1}$.
\end{lem}
\begin{proof} Following the ideas of \cite[Ch.~4.2]{Ladyzhenskaya1985} we can transfer the results from the scalar to the vectorial wave equation.
	 First we integrate \eqref{ieqn:sec_order} over the interval $(0,t)$ and use the Fubini theorem and the fundamental theorem of calculus. Then we can estimate by use of Assumption~\ref{A:Assumptions_all} to get
		\begin{align*}
			\int\limits_{\Omega} \left[\varepsilon_{\min}(\partial_tu)^2(t,x)+(\mu^{\max})^{-1}(\mathrm{curl}_x u)^2(t,x)\right]\ dx& \leq  \int\limits_{\Omega} \left[\varepsilon^{\max}(\partial_tu)^2(0,x)+(\mu_{\min})^{-1}(\mathrm{curl}_x u)^2(0,x)\right]\ dx \\
			&+2\int\limits_0^t\int\limits_{\Omega} j(s,x )\partial_tu(s,x)\ dx\ ds.
		\end{align*}
		Using $(a+b)^2\leq 2a^2+2b^2$ and Cauchy-Schwarz leads for $s>0$ to 
		\begin{align*}
			\norm{u(s,.)}_{L^2(\Omega)}^2&=\norm{u(0,.)+\int\limits_0^s1\cdot \partial_t u(t,.)\ dt}_{L^2(\Omega)}^2\\
			&\leq 2 \norm{u(0,.)}_{L^2(\Omega)}^2+2s \norm{u}_{L^2((0,s)\times\Omega)}^2+\norm{\partial_t u}_{L^2((0,s)\times\Omega)}^2+ \norm{\mathrm{curl}_x u}_{L^2((0,s)\times\Omega)}^2.
		\end{align*}
		By adding both equations together we get for $
		z(t)=\norm{u(t,.)}_{L^2(\Omega)}^2 + \norm{\partial_t u(t,.)}_{L^2(\Omega)}^2+\norm{\mathrm{curl}_x u(t,.)}_{L^2(\Omega)}^2 $
		the inequality
		\begin{align*}
			\min(1,\varepsilon_{\min},(\mu^{\max})^{-1})z(t)\leq \max(2,\varepsilon^{\max},(\mu_{\min})^{-1})z(0)+2\int\limits_0^t\|j\|_{L^2(\Omega)} z^{1/2}(s)\ ds+2t\int\limits_0^tz(s)\ ds.
		\end{align*}
		for all $t\in [0,T]$. Then we write
		\begin{align*}
			a:=\max(2,\varepsilon^{\max},(\mu_{\min})^{-1}),
			b:=\min(1,\varepsilon_{\min},(\mu^{\max})^{-1})
		\end{align*}
		and define $\hat{z}(t) = \max\limits_{0\leq \xi\leq t} z(\xi)$	to get
		\begin{align}
			\label{ieqn:solv_i1}
			b\hat{z}(t)\leq az(0)+2\|j\|_{L^{2,1}(Q)} \hat{z}^{1/2}(t)+2t^2\hat{z}(t).
		\end{align}
		After solving this inequality while keeping in mind that $\hat{z}^{1/2}(t)\geq0$ we derive 
		\begin{align*}
			\hat{z}^{1/2}(t)\leq  \frac{2\|j\|_{L^{2,1}(Q)} + \sqrt{4\|j\|_{L^{2,1}(Q)}^2+4(b-2t^2) az(0)}}{2(b-2t^2)}.
		\end{align*}
		By using $\sqrt{a^2+b^2}\leq (a+b)$, for $a,b\geq 0$, and considering $t\leq \min(\sqrt{\frac{b}{4}},T)$ we get that
		\begin{align*}
			\hat{z}^{1/2}(t)&\leq \frac{2+2}{b}\|j\|_{L^{2,1}(Q)}+\sqrt{\frac{2}{b}}\sqrt{a}z^{1/2}(0)
		\end{align*}
		since $2(b-2t^2)\geq b$. If $\sqrt{\frac{b}{4}}<T$, then we are done. Otherwise we choose as initial value $t_1=\sqrt{\frac{b}{4}}$, instead of $t=0$ and repeat the computations from above to get the desired result.
\end{proof}

\begin{cor}\label{cor:tool1_sigma}
	Let the Assumption~\ref{A:Assumptions_all} hold true. Let $\vec A\in H^2(0,T;H(\cu;\Omega))$ with 	$\vec A|_{t=0} =  \vec \phi$ and $\partial_t \vec  A|_{t=0} =\vec \psi$  be the solution of
	 {\begin{align*}
			(\varepsilon \partial_{tt} \vec A(t) , \partial_t \vec A(t))_{L^2(\Omega)}+(\sigma \partial_{t} \vec A(t), \partial_t \vec A(t))_{L^2(\Omega)}+(\mu^{-1} \cu_x  \vec A(t) , \cu_x \partial_t \vec A(t))_{L^2(\Omega)} = (\vec j_a(t),\partial_t \vec A(t))_{L^2(\Omega)}
		\end{align*}
		for each $t\in(0,T)$ and } $\vec j_a\in L^{1}(0,T;L^2(\Omega;\R^d))$. Then we derive that
	\begin{align*}
		z^{1/2}(t) := \left(\int\limits_{\Omega}\left(\abs{\vec A}^2+\abs{\partial_t \vec A}^2 +\abs{\cu_x \vec A}^2\right)(t,x)\, \mathrm dx\right)^{1/2} \leq c_2(t) z^{1/2}(0)+c_3(t) \|\vec j_a\|_{2,1,Q},
	\end{align*}
	for all  $t\in [0,T]$. The constants $c_2$ and $c_3$  depend on  $ \varepsilon_{\min},(\mu^{\max})^{-1},\varepsilon^{\max}$, $(\mu_{\min})^{-1}$ as well as $T$.
\end{cor}
\begin{proof}
	Using the positive semi-definiteness of $\sigma$ we compute
	 \begin{align*}
			(\varepsilon \partial_{tt} \vec A(t) , \partial_t \vec A(t))_{L^2(\Omega)}+(\mu^{-1} \cu_x  \vec A(t) , \cu_x \partial_t \vec A(t))_{L^2(\Omega)} \leq (\vec j_a(t),\partial_t \vec A(t))_{L^2(\Omega)}.
		\end{align*}
		for each $t\in(0,T)$. Therefore we can use Lem.~\ref{lem:H2_ungl_not_var} to prove the desired estimate.
\end{proof}

Now, we are ready to prove the existence of the solution of the variational formulation \eqref{vf:H1VarForm}.

\begin{prop}[Existence]\label{prop:Existence_H1_drhs}
	Let Assumption~\ref{A:Assumptions_all} hold true. Then there exists a solution of the variational formulation \eqref{vf:H1VarForm}.
\end{prop}
\begin{proof}
	Let us consider $\omsig := \Omega\cap \mathrm{supp}(\sigma)$ which is a Lipschitz domain by Assumption~\ref{A:Assumptions_all}. The main idea of this proof is to split the differential equation into two equations over different domains, namely one domain is the support $\omsig$ of the conductivity $\sigma$ and the other $\Omega\backslash\overline{\omsig}$. Then we add the solutions together to show existence. The produced solution is the solution to an interface problem with zero tangential traces on the interface $\partial \omsig \backslash \partial \Omega$.  Hence, let us take a look at the resulting equations on the domain $\omsig$.\\
	1. On the domain $\omsig$ we now use a Galerkin method. 	Let us define the bilinear  {form}
	\begin{align}
		\label{def:BilinearForm2}	
		a_{\omsig}(\vec A(t) ,\vec \phi)\coloneqq (\varepsilon\partial_{tt}\vec A(t) ,\vec\phi)_{L^2(\omsig)} +(\sigma\partial_t \vec A(t) ,\vec \phi)_{L^2(\omsig)} +(\mu^{-1}\cu_x  \vec A(t) ,\cu_x \vec\phi)_{L^2(\omsig)},
	\end{align}
	for $\vec \phi\in H(\cu;\omsig)$, $t\in(0,T)$.  We consider $\{\vec\varphi_k\}_{ {k \in \Z}}$, the fundamental system of $H_0(\cu;\omsig)$ of Lem.~\ref{lem:FS-H_0curl} for $\Omega = \Omega_\sigma$. Let $N\in\N$. Using the basis representation \eqref{eq:BasisRepresentation}  we then search for a
	\begin{equation}
		\label{def:AN}
		\vec A^N(t)=\sum\limits_{k=-N}^N  { {a_k^N}}(t)\vec \varphi_k(x)
	\end{equation}
	which  solves
	\begin{align}
		\nonumber
		a_{\omsig}(\vec A^N(t) ,\vec \varphi_l)	&=(\vec j_a {(t)},\vec \varphi_l)_{L^2(\omsig)},\\
		\label{eqn:ckN_drhs}
		\frac{d}{dt}  {a_k^N}(t)&=(\vec \psi,\vec \varphi_k)_{L^2(\omsig)},\\
		\nonumber
		 {a_k^N}(0) &=\alpha_k^N,
	\end{align}
	for all $l,k \in \Z_N \coloneqq \{ z\in \Z: |z| \leq N\}$,	where $\alpha_k^N$ are the coefficients of
	\begin{align*}
		\vec \phi^N(x) = \sum\limits_{k=-N}^N \alpha_k^N\vec \varphi_k(x),
	\end{align*}
	and $\vec \phi^N\to \vec  \phi$ in $H_0(\cu;\omsig)$ for $N\to \infty$. We have $\vec A^N(0,x) = \vec \phi^N(x)$ and define $f_k:= (\vec j_a,\vec \varphi_k)_{L^2(\omsig)} $. 
	
	Since $\sigma$ is uniformly positive definite and bounded over $\omsig$, the induced weighted scalar product $(\sigma . ,.)_{L^2(\omsig)}$ is equivalent to $(\varepsilon . ,.)_{L^2(\omsig)}$. Hence, there exists a $\beta_k\in\R_+$ such that
	\begin{align*}
		(\sigma \vec \varphi_k,\vec \varphi_l)_{L^2(\omsig)} = \beta_k \delta_{kl}
	\end{align*}
	for $k,l\in \Z_N$.	These $\beta_k$ are bounded from below by $\sigma_{\min}$.
	
	Next, we combine everything to arrive 	for $t\in (0,T)$ and $k = 0,\dots,N$ at 
	\begin{align}
		\nonumber
		 {a_k''^N}(t)+ \beta_k  {a_k'^N}(t)+ \lambda_k  {a_k^N}(t) &= f_k(t),\\
		\label{eqn:ck1}
		 {a_k'^N}(0)  &= (\vec \psi,\vec \varphi_k)_{L^2(\omsig)},\\
		\nonumber
		 {a_k^N} (0) &=\alpha_k^N
	\end{align}
	where $\vec \varphi_k$  is an eigenfunction of $\cu_x\mu^{-1}\cu_x$ in $\Omega_\sigma$, and for $k=-N,\dots,-1$
	\begin{align}
		\nonumber
		 {a_k''^N}(t)+ \beta_k  {a_k'^N}(t) &= f_k(t),\\
		\label{eqn:ck2}
		 {a_k'^N}(0) &= (\vec \psi,\vec \varphi_k)_{L^2(\omsig)},\\
		\nonumber
		 {a_k^N}(0) &= {\alpha_k^N}
	\end{align}
	where $\vec \varphi_k$  is part of the kernel of $\cu_x\mu^{-1}\cu_x$. The solutions can be computed using standard techniques for ordinary differential equations such as \cite[L.~20, L.~21]{tenenbaum1985ordinary}.  The solutions are well defined for $f_k\in L^1(0,T)$. Therefore, there exists for $\vec j_a\in L^1(0,T;L^2(\omsig))$ a solution $\vec A^N\in C^2(0,T;H(\cu;\omsig))$. If we multiply  (\ref{eqn:ckN_drhs}) with $c_l'^N(t)$ and sum up over $l=-N,\dots,N$, we compute using \eqref{def:BilinearForm2}
	\begin{align*}
		a_{\omsig}(\vec A(t) ,\partial_t \vec A(t)) =(\vec j_a(t),\partial_t\vec  A^N(t))_{L^2(\omsig)},
	\end{align*}
	for $t\in[0,T]$. This satisfies the conditions to apply Cor.~\ref{cor:tool1_sigma} with which we arrive at
	\begin{align*}
		\left(\int\limits_{\omsig}\abs{\vec A^N(t)}^2+\abs{\partial_t \vec A^N(t)}^2 +\abs{\cu_x \vec A^N(t)}^2\,\mathrm dx\right)^{1/2} \leq c_2(t) (z^{N})^{ 1/2}(0)+c_3(t) \|\vec j_a\|_{L^1(0,T;L^2(\omsig))}
	\end{align*}
	for every $t\in(0,T)$.
	By construction  $c_2$ and $c_3$ are monotonically increasing and therefore bounded by $c_2(T)$ respectively $c_3(T)$. Additionally we get with Bessel's inequality , \cite[Thm. 1.7.1]{aubin}, in $H(\cu;\omsig)$ and $L^2(\Omega_\sigma)$
	\begin{align*}
		z^N(0) &= \int\limits_{\omsig}\abs{\vec \phi^{N}}^2+\abs{\vec \psi^N}^2 +\abs{\cu_x \vec \phi^N}^2\,\mathrm dx\leq c\|\vec \phi\|^2_{H(\cu;\omsig)} + c \|\vec \psi\|^2_{L^2(\omsig)}
	\end{align*}
	We define $Q_\sigma := [0,T]\times \omsig$. Therefore
	\begin{align*}
		\|\vec A^N\|_{H^{\cu;1}(Q_\sigma)} \leq c\|\vec \phi\|^2_{H(\cu;\omsig)}+ \tilde{c}\|\vec \psi\|^2_{L^2(\omsig)} +\hat{c} \|j_a\|_{L^1(0,T;L^2(\omsig))} < C,
	\end{align*}
	where $C$ is independent of $N$.  
	
	Now, we want to transfer the existence to the space-time variational formulation \eqref{vf:H1VarForm}.  We split $\vec A^N(t,x)=\vec A_0^N(t,x)+\vec \phi^N(x)$. Since $\|\vec \phi^N\|_{H(\cu;\omsig)}$ is bounded by $\|\phi\|_{H(\cu;\omsig)}$ because of Bessel's inequality, the sequence $(\vec A_0^N)_{N\in\N}$ is bounded as well. The space \[H^{\cu;1}_{0;}(Q_\sigma) = H^1(0,T;L^2(\omsig))\bigcap L^2(0,T;H_0(\cu;\omsig))\]
	is a Hilbert space, see Lem.~\ref{lem:space_time_HR}. Hence there exists a weakly convergent subsequence of $(\vec A_0^N)_{N\in\N}$. We write this subsequence as  $\{\vec A_0^N \}_N$ for convenience. Then there exists a $\vec A_0\in H^{\cu;1}_{0;0,}(Q_\sigma)$ with
	\begin{align*}
		\vec A_0^N&\rightharpoonup \vec A_0 &&\textmd{ in }L^2(Q_\sigma),\\
		\partial_t \vec  A_0^N &\rightharpoonup \partial_t \vec A_0 &&\textmd{ in }L^2(Q_\sigma),\\
		\cu_x  \vec  A_0^N &\rightharpoonup \cu_x  \vec A_0 &&\textmd{ in }L^2(Q_\sigma).
	\end{align*}
	
	We constructed $\vec A_0^N$ such that
	\begin{align*}
		a_{\omsig}(\vec A_0^N(t) , \vec \varphi_l) =(\vec j_a(t) ,\vec \varphi_l)_{L^2(\Omega_\sigma )}-(\mu^{-1}\cu_x \vec \phi^N,\cu_x \vec \varphi_l)_{L^2(\Omega_\sigma)}
	\end{align*}
	for all $l=-N,..,N$, $t\in(0,T)$. Let $M\in \N$. Choose $N>M$ and $d_l\in H^1(0,T)$ where $d_l(T)= 0$, $l=-M,\dots,M$. Multiply the equation with $d_l$ and sum up over $l=-M,\dots,M$.
	Moreover, we get for $\vec \eta(t,x):= \sum\limits_{l=-M}^M d_l(t )\vec \varphi_l(x)$ and  $t\in(0,T)$ the equation
	\begin{align*}
		a_{\omsig}( \vec A_0^N(t) ,\vec \eta) =(\vec j_a(t) ,\vec \eta)_{L^2(\omsig)}-(\mu^{-1}\cu_x \vec \phi^N,\cu_x\vec  \eta)_{L^2(\omsig)}.
	\end{align*}
	By integration over $(0,T)$ and using integration by parts for the first term we get to the bilinear form \eqref{def:BilinearForm} and
	\begin{align*}
		a_{Q_\sigma}(\vec A_0^N,\vec \eta)=(\vec j_a,\vec \eta)_{L^2(Q_\sigma )} - (\varepsilon\partial_t \vec A^N|_{t=0},\vec \eta|_{t=0})_{L^2(\omsig)}-( \mu^{-1}\cu_x \vec \phi^N, \cu_x\vec \eta)_{L^2(Q_\sigma)} .
	\end{align*}
	Next, we take the limit $N\to\infty$. Since $\vec\eta,\partial_t\vec  \eta$ and $\cu_x \vec \eta$ are in $L^2(Q_\sigma;\R^d)$ we get with the weak convergence that
	\begin{align*}
		a_{Q_\sigma}(\vec A_0,\vec \eta) =(\vec j_a,\vec \eta)_{L^2(Q_\sigma)}-(\varepsilon\vec \psi,\vec \eta(0,.))_{L^2(\omsig)} -( \mu^{-1}\cu_x \vec \phi, \cu_x\vec \eta)_{L^2(Q_\sigma)}.
	\end{align*}
	This equation holds for all $\eta$ with the representation $\sum\limits_{l=-M}^M d_l(t )\vec \varphi_l(x)$. Let $\mathcal{M}_M$ be the space of such functions. For every $M\in \N$ we can repeat this argumentation. The space $\bigcup\limits_{M=1}^\infty\mathcal{M}_M$ is dense in $L^2(0,T;H_0(\cu;\Omega_\sigma))$ because we can approximate every element with such a sum, see  \cite[Prop.~23.2d]{zeidler}. Hence the equation above holds true for every $\eta\in  H^{\cu;1}_{0;,0}(Q_\sigma)$ and therefore $\vec A$  is the weak solution of our differential equation in $H^{\cu;1}_{0;0,}(Q_\sigma)$. 
	\vspace*{.2cm}\\
	2. Let us now consider the spatial domain  $\Omega_0:=\Omega\backslash\overline{\omsig}$  and space-time domain $Q_0:=(0,T)\times \Omega_0 $ where $\sigma$ is zero.
	 Again we consider the ansatz \eqref{def:AN} and the equation
		\begin{align}
			\nonumber
			(\varepsilon\partial_{tt}\vec A^N(t) ,\vec\varphi_l)_{L^2(\Omega_0)}  +(\mu^{-1}\cu_x  \vec A^N(t) ,\cu_x \vec\varphi_l)_{L^2(\Omega_0)}	&=(\vec j_a(t),\vec \varphi_l)_{L^2(\Omega_0)},\\
			\label{eqn:ckN_drhs_sigma0}
			\frac{d}{dt} a_k^N(t)&=(\vec \psi,\vec \varphi_k)_{L^2(\omsig)},\\
			\nonumber
			a_k^N(0) &=\alpha_k^N
		\end{align}
		for $l,k \in \Z_N$.
		Then we follow the same steps as before and we consider instead of the equations \eqref{eqn:ck1} and \eqref{eqn:ck2} the equations
		\begin{align}
			\nonumber
			a_k''^N(t)+ \lambda_k a_k^N(t) &= f_k(t),\\	
			\label{eqn:ode2_sigma_equal_0}
			a_k'^N(0) &= (\vec \psi,\vec \varphi_k)_{L^2(\Omega_0)},\\
			\nonumber
			a_k^N (0) &=\alpha_k^N
		\end{align}
		for $t\in (0,T)$ and $\lambda_k>0$ for  $k=0,\dots,N$, i.e. $\vec \varphi_k$  is an eigenfunction of $\cu_x\mu^{-1}\cu_x$ in $\Omega_0$, and $\lambda_k=0$ for $k=-N,\dots,-1$ 
		where $\vec \varphi_k$  is part of the kernel of $\cu_x\mu^{-1}\cu_x$ in $\Omega_0$. The ordinary equation \eqref{eqn:ode2_sigma_equal_0} is uniquely solvable for $f_k\in L^1(0,T)$ and can be solved using standard techniques such as \cite[L.~20, L.~21]{tenenbaum1985ordinary}. By multiplying \eqref{eqn:ckN_drhs_sigma0} with  $c_l'^N(t)$ and summing up over $l=-N,\dots,N$, we compute using \eqref{def:BilinearForm2}
		\begin{align*}
			(\varepsilon\partial_{tt}\vec A^N(t) ,\vec A^N(t))_{L^2(\Omega_0)}  +(\mu^{-1}\cu_x  \vec A^N(t) ,\cu_x \vec A^N(t))_{L^2(\Omega_0)}	&=(\vec j_a(t),\vec A^N(t))_{L^2(\Omega_0)}.
		\end{align*}
		This satisfies the conditions to apply Lem.~\ref{lem:H2_ungl_not_var} with which we arrive at
		\begin{align*}
			\left(\int\limits_{\Omega_0}\abs{\vec A^N(t)}^2+\abs{\partial_t \vec A^N(t)}^2 +\abs{\cu_x \vec A^N(t)}^2\,\mathrm dx\right)^{1/2} \leq c_2(t) (z^{N})^{ 1/2}(0)+c_3(t) \|\vec j_a\|_{L^1(0,T;L^2(\Omega_0))}
		\end{align*}
		for every $t\in(0,T)$ and 	
		$z^N(0) = \|\vec \phi^{N}\|_{L^2(\Omega_0)}^2+\norm{\vec \psi^N}_{L^2(\Omega_0)}^2 +\norm{\cu_x \vec \phi^N}_{L^2(\Omega_0)}^2 \leq c\|\vec \phi\|^2_{H(\cu;\Omega_0)} + c \|\vec \psi\|^2_{L^2(\Omega_0)}$.
		Therefore we have weak convergence of $(\vec A^N)_N$, $(\cu_x \vec A^N)_N$ and $(\partial_t \vec A^N)_N$ in $L^2(Q)$. By going through the same argumentation as above, we see that the limit $\vec A\in H^{\cu;1}_{0;0,}(Q_0)$ indeed is the weak solution of
		\begin{align*}
			-(\varepsilon\partial_{t}\vec A ,\partial_{t}\vec  v)_{L^2(Q_0)}  +(\mu^{-1}\cu_x  \vec A^N ,\cu_x \vec v)_{L^2(Q_0)}	&=(\vec j_a(t),\vec v)_{L^2(Q_0)}- \left(\varepsilon\vec \psi, \vec v(0,\cdot)\right)_{L^2(\Omega_0)},
		\end{align*}
		with  $\vec A(0,.)=\vec \phi$. Adding both solutions on the domains $Q_0$ and $Q_\sigma$ yields existence. 
	
\end{proof}

\subsection{Uniqueness}

Next, we take a look at the uniqueness of  the solution to the variational formulation \eqref{vf:H1VarForm} with the bilinear form \eqref{def:BilinearForm} and right-hand side \eqref{def:LinearForm}.
\begin{prop}[Uniqueness]\label{prop:Uniqueness_H1_drhs}
	Let Assumption~\ref{A:Assumptions_all} hold true. Then there exists a unique solution for the variational formulation \eqref{vf:H1VarForm}.
\end{prop}
\begin{proof}
	In Proposition~\ref{prop:Existence_H1_drhs} we have already shown existence. What is left to be proven is the uniqueness.
	Assume that there are two solutions $\vec A'$ and $\vec A''$, then $\vec w:= \vec A'-\vec A''$ satisfies the variational formulation
	\begin{align*}
		- (\varepsilon\partial_t \vec w ,\partial_t \vec v)_{L^2(Q)}+( \sigma\partial_t \vec w , \vec v)_{L^2(Q)}+ ( \mu^{-1}	\cu_x \vec w, \cu_x \vec v)_{L^2(Q)} = 0 
	\end{align*}
	for all $\vec v\in H^{\cu;1}_{0;,0}(Q) $. Additionally $\vec w(0,x) = 0$ holds true for $x\in \Omega$. Choose  $b\in [0,T]$ arbitrary and consider
	\begin{align*}
		\vec \eta(t,x) := \begin{cases}	
			\int\limits_b^t\vec w(\tau,x)\,\mathrm d\tau, \hspace*{.3cm}&\textmd{ for } 0\leq t\leq b,\\
			0, \hspace*{2cm}&\textmd{ for } b\leq t\leq T.
		\end{cases}
	\end{align*}
	Then $\vec \eta\in H^{\cu;1}_{0;,0}(Q)$ with  $\vec \eta(t,x) = 0$ for $t\geq b$ and we get for $\vec v=\vec \eta$:
	\begin{align*}
		- (\varepsilon\partial_t \vec w,\partial_t \vec \eta)_{L^2(Q_{(0,b)})}+(\sigma\partial_t \vec w, \vec \eta)_{L^2(Q_{(0,b)})}+  (\mu^{-1}	\cu_x \vec  w, \cu_x\vec  \eta)_{L^2(Q_{(0,b)})} = 0, 
	\end{align*}
	where $Q_{(0,b)}$ is the intersection of $Q$ with the half space $t< b$. Since $\partial_t \vec \eta (t,x)= w(t,x) $	for $(t,x)\in {Q_{(0,b)}}$, we compute
	\begin{align*}
		(\varepsilon\partial_{tt}\vec \eta, \partial_t\vec \eta)_{L^2(Q_{(0,b)})} - (\sigma\partial_{tt}\vec \eta,\vec \eta)_{L^2(Q_{(0,b)})} - (\mu^{-1}	\cu_x \partial_t \vec \eta,\cu_x \vec \eta)_{L^2(Q_{(0,b)})} = 0.
	\end{align*}
	Through integration by parts we get
	\begin{align*}
		- (\sigma\partial_{tt}\vec \eta, \vec \eta)_{L^2(Q_{(0,b)})}&= (\sigma\partial_{t}\vec \eta,\partial_{t} \vec \eta)_{L^2(Q_{(0,b)})}\\
		&\geq  0,
	\end{align*}
	because $\partial_t\vec \eta(0,x) =\vec w(0,x) = 0$ for $x\in\Omega$, $\vec \eta(b,x) = 0$ by definition and $\sigma$ is positive semi-definite. Therefore 
	\begin{align*}
		(\varepsilon\partial_{tt}\vec \eta,\partial_t\vec  \eta)_{L^2(Q_{(0,b)})} - (\mu^{-1}	\cu_x \partial_t\vec  \eta, \cu_x \vec \eta)_{L^2(Q_{(0,b)})} \leq 0
	\end{align*}
	holds true. Hence
	\begin{align*}
		\frac{1}{2}\int\limits_{Q_{(0,b)}} \partial_t(\varepsilon\partial_{t}\vec \eta\cdot\partial_{t}\vec \eta ) \,\mathrm dx\,\mathrm dt- \frac{1}{2} \int\limits_{Q_{(0,b)}} \partial_t(\mu^{-1}\cu_x  \vec \eta\cdot\cu_x  \vec \eta )\,\mathrm dx\,\mathrm dt \leq 0.
	\end{align*}
	holds true  because $\varepsilon$ and $\mu^{-1}$ are symmetric. From the definition of $\vec \eta$ we compute $\vec \eta(b,x)= 0$ and therefore $(\cu_x\vec \eta)(b,x) = 0$  for $x\in \Omega$. Additionally, $\partial_t\vec \eta(0,x) =\vec w(0,x) = 0$ for $x\in\Omega$  and therefore we arrive at
	\begin{align*}
		\int\limits_{\Omega} (\varepsilon\partial_{t}\vec \eta\cdot\partial_{t}\vec \eta) (b,x)\,\mathrm dx 
		+  \int\limits_{\Omega} (\mu^{-1}\cu_x \vec \eta\cdot\cu_x \vec \eta )(0,x)\,\mathrm dx\leq 0.
	\end{align*}
	From this we derive by  $\partial_t\vec \eta(b,x) =\vec w(b,x)$ that
	\begin{align*}
		\int\limits_{\Omega} \vec w^2 (b,x)\,\mathrm dx = 0\qquad \text{and}\qquad	\int\limits_{\Omega}\left( \int_0^b \cu_x\vec w \right)^2 (x)\,\mathrm dx = 0
	\end{align*}
	holds true for any $b\in(0,T)$, since $\partial_t \vec\eta (t,x)= \vec w(t,x)$. With this, we can deduce that $\vec w(t,x)$ vanishes almost everywhere.	
\end{proof}

\subsection{Norm estimate}
Now that we know that the variational formulation \eqref{vf:H1VarForm} is uniquely solvable, we take a look at the norm estimate. In this part of the section, we consider $\vec j_a\in L^2(Q;\R^d)$. We will derive the norm estimate of Theorem \ref{thm:ExistenceAndUniqueness} and take a closer look at the dependencies of the coefficients  $c_\phi$, $c_\phi^c$, $c_\psi$,  and $c_f$.

\begin{lem}\label{lem:ODE2}
	Let Assumption~\ref{A:Assumptions_all} hold, $\vec j_a\in L^2((0,T) \times \omsig;\R^d)$,  $\{\lambda_k,\varphi_k\} _{k\in \N_0}$ the pair of non-zero eigenvalues and eigenfunctions of the $\cu_x \mu^{-1} \cu_x$-operator in $\omsig$ from Lemma~\ref{lem:FS-H_0curl}, $\alpha_k\in\R$ and 	$\beta_k:= (\sigma \vec \varphi_k,\vec \varphi_k)_{L^2(\omsig)} >0$ for $k \in \N_0$. Additionally, let $ {a_k}$ be the solution of the ordinary differential equation 
	\begin{align}
		\nonumber
		 {a_k}''(t)+ \beta_k  {a_k}'(t)+ \lambda_k  {a_k} (t) &= f_k(t),\\
		\label{eqn:ODE2}
		 {a_k} (0) &=\alpha_k,\\
		\nonumber
		 {a_k}'(0) &= (\vec \psi,\vec \varphi_k)_{L^2(\Omega_\sigma)}	
	\end{align}
	for $t\in (0,T)$, and $f_k(t) := (\vec j_a(t), \vec \varphi_k)_{L^2(\omsig)}$. Then there exist positive constants $c_\alpha^\lambda$, $c_\alpha$, $c_\psi$ and $c_f$ such that
	\[
	\sum_{k\in \N_0}\intT \lambda_k( {a_k})^2(t)\, \mathrm dt +\intT (  {a_k}')^2(t)\,\mathrm dt \leq \sum_{k\in \N_0} c_\psi (\vec \psi,\vec \varphi_k)^2_{L^2(\omsig)}+(c_\alpha^\lambda \lambda_k +c_\alpha ) (\alpha_k)^2 + c_f T\|f_k\|^2_{L^2(0,T)}.
	\]
\end{lem}

\begin{proof}
	For these estimates we need to consider three cases, namely $\beta_k^2-4\lambda_k>0$, $\beta_k^2-4\lambda_k<0$ and $\beta_k^2-4\lambda_k=0$.\\
	1. Let $\beta_k^2-4\lambda_k>0$. We define $\kappa_1 :=\frac{-\beta_k + \sqrt{\beta_k^2-4\lambda_k}}{2}$,
	$ \kappa_2 :=\frac{-\beta_k - \sqrt{\beta_k^2-4\lambda_k}}{2}$, $\gamma_k = \sqrt{\beta_k^2-4\lambda_k}$.  Then the solution of \eqref{eqn:ODE2} is given by
	\begin{align*}
		 {a_k} (t) = &(\vec \psi,\vec \varphi_k)_{L^2(\omsig)}\frac{ e^{\kappa_1t}- e^{\kappa_2t}}{\gamma_k}+\alpha_k\frac{ \kappa_1e^{\kappa_2t}- \kappa_2e^{\kappa_1t}}{\gamma_k}\\
		&+\frac{1}{\gamma_k}\int\limits_0^t	( e^{\kappa_1(t-s)}-e^{\kappa_2(t-s)})f_k(s)\,\mathrm ds.
	\end{align*}
	We know that $\kappa_1,\kappa_2<0$, because $\beta_k^2-4\lambda_k>0$ and so $(-\beta_k-\sqrt{\beta_k^2-4\lambda_k})<0$. Additionally we get $-(\kappa_1+\kappa_2) = \beta_k$ as well as $\kappa_1\kappa_2 = \lambda_k$. With these facts and using $(a+b)^2\leq 2 a^2+2b^2$ we compute  
	\[
	\intT  \left(e^{\kappa_1t}- e^{\kappa_2t}\right)^2\,\mathrm dt\leq \frac{1}{\kappa_1}(e^{2\kappa_1 T}-1)+ \frac{1}{\kappa_2}(e^{2\kappa_2 T}-1)\leq \frac{1}{-\kappa_1}+\frac{1}{-\kappa_2} = \frac{\beta_k}{\lambda_k}.
	\]
	In the same way, we estimate
	\[
	\intT  \left(\kappa_1 e^{\kappa_1t}- \kappa_2e^{\kappa_2t}\right)^2\,\mathrm dt\leq \kappa_1(e^{2\kappa_1 T}-1)+\kappa_2(e^{2\kappa_2 T}-1)\leq \beta_k.
	\]
	With this, we can compute an estimate for the desired norms  using again $(a+b)^2\leq 2 a^2+2b^2$
	\begin{align*}
		\intT \lambda_k( {a_k})^2(t)\,\mathrm dt+\intT (  {a_k}')^2(t)\,\mathrm dt\leq&\, 2(\vec \psi,\vec \varphi_k)^2_{L^2(\omsig)}\frac{4 \beta_k}{\beta_k^2-4\lambda_k}+2(\alpha_k)^2\lambda_k\frac{4\beta_k}{\beta_k^2-4\lambda_k}\\
		&+2\frac{2\beta_k}{\beta_k^2-4\lambda_k}T\intT f_k^2(s)\,\mathrm ds.
	\end{align*}
	From the estimate $\sqrt{2ab}\leq (a-b)$ for $a,b\leq 0$ and $a>b$ we derive $ \frac{1}{(\beta_k^2-4\lambda_k)}\leq \frac{1}{\sqrt{2}b_k2\sqrt{\lambda_k}}\leq\frac{1}{2b_k\sqrt{\lambda_k}}$. Using this estimate we arrive at 
	\begin{align*}
		\intT \lambda_k( {a_k})^2(t)\,\mathrm dt+\intT (  {a_k}')^2(t)\,\mathrm dt\leq (\vec \psi,\vec \varphi_k)^2_{L^2(\omsig)}\frac{4 }{\sqrt{\lambda_0}}+(\alpha_k)^2\lambda_k\frac{4 }{\sqrt{\lambda_0}}+\frac{2 }{\sqrt{\lambda_0}}T\| f_k\|^2_{L^2(0,T)}
	\end{align*}
	where $\lambda_0$ is the smallest eigenvalue of $\{\lambda_k\}_{k\in \N_0}$ from Lemma \ref{lem:FS-H_0curl} for $\Omega=\omsig$. \\
	2. Let $\beta_k^2-4\lambda_k<0$ and define $\gamma_k = \sqrt{4\lambda_k-\beta_k^2}$. Then we write the solution of \eqref{eqn:ODE2} as 
	\begin{align*}
		 {a_k} (t) =&\ \alpha_k e^{-\frac{\beta_k}{2}t}\left( \cos(\gamma_kt)+\frac{\beta_k}{2\gamma_k} \sin(\gamma_kt)\right) +\frac{1}{\gamma_k}(\vec \psi,\vec \varphi_k)_{L^2(\omsig)}  e^{-\frac{\beta_k}{2}t}\sin(\gamma_kt)\\
		&+\frac{1}{\gamma_k}\int\limits_0^te^{\frac{\beta_k}{2}(s-t)}\sin(\gamma_k(t-s))f(s)\,\mathrm ds
	\end{align*}
	Note that the parameter $\beta_k = (\sigma \vec \varphi_k,\vec \varphi_k)_{L^2(\omsig)}$ is positive. Hence we can estimate 
	\[\intT e^{-\beta_k t}\,\mathrm dt\leq \frac{1}{\beta_k}. \]
	Next, we use the fact that $\sin^2(\gamma_kt)\leq 1$ and $\cos^2(\gamma_kt)\leq 1$. Then we derive by using $(a+b)^2\leq 2 a^2+2b^2$ and the Cauchy-Schwarz inequality yet again the following estimate
	\begin{align*}
		\intT \lambda_k( {a_k})^2(t)\,\mathrm dt\leq \lambda_k(\vec \psi,\vec \varphi_k)^2_{L^2(\omsig)}\frac{4 }{\beta_k\gamma_k^2}+4(\alpha_k)^2\lambda_k\left(\frac{1}{\beta_k}+\frac{\beta_k}{4\gamma_k^2}\right)+\frac{2\lambda_k}{\beta_k\gamma_k^2}T\intT f_k^2(s)\,\mathrm ds.
	\end{align*}
	In the same way, we estimate
	\begin{align*}
		\intT ( {a_k}')^2(t)\,\mathrm dt\leq\ & (\vec \psi,\vec \varphi_k)^2_{L^2(\omsig)}\frac{4 }{\beta_k\gamma_k^2} \left(\frac{\beta_k^2}{2}+2\gamma_k^2\right) + 4(\alpha_k)^2\frac{1}{\beta_k}\left(\gamma_k^2 + \frac{\beta_k^4}{16\gamma_k^2}\right)\\
		&+\frac{2}{\beta_k\gamma_k^2}\left(2\frac{\beta_k^2}{4}+ 2 \gamma_k^2\right)T\intT f_k^2(s)\,\mathrm ds.
	\end{align*}
	By inserting $\gamma_k^2 = {4\lambda_k-\beta_k^2}$ we get
	\begin{align*}
		\intT \lambda_k( {a_k})^2(t)\,\mathrm dt+\intT ( {a_k}')^2(t)\,\mathrm dt\leq\ & (\vec \psi,\vec \varphi_k)^2_{L^2(\omsig)}\left(\frac{36\lambda_k }{\beta_k(4\lambda_k-\beta_k^2)}+\frac{6\beta_k }{4\lambda_k-\beta_k^2} \right)\\
		& + 4(\alpha_k)^2\left(\frac{5\lambda_k}{\beta_k}+\frac{\lambda_k\beta_k}{4(4\lambda_k-\beta_k^2)}+ \frac{\beta_k^3}{16(4\lambda_k-\beta_k^2)}\right)\\
		&+\left(\frac{2\lambda_k}{\beta_k(4\lambda_k-\beta_k^2)}+\frac{\beta_k}{4\lambda_k-\beta_k^2}+\frac{4}{\beta_k}\right)T\intT f_k^2(s)\,\mathrm ds.
	\end{align*}
	Note that 	the term $\frac{\lambda_k}{4\lambda_k-\beta_k^2}$ only shows up when $\beta_k^2<4\lambda_k$. It is bounded since $\beta_k$ is bounded by $\beta_{\max} := \max_k \beta_k$, which is bounded by $\sup_{x\in\omsig} \sigma(x)$, but $\lambda_k$ is increasing monotonically. Therefore the maximum will be reached for smaller $k$.\\
	Let us define $\beta_{\min} := \min_k \beta_k>0$. Then we use the estimate $ \frac{1}{(\beta_k^2-4\lambda_k)}\leq \frac{1}{2b_k\sqrt{\lambda_k}} $ to derive
	\begin{align*}
		\intT \lambda_k( {a_k})^2(t)\,\mathrm dt+\intT ( {a_k}')^2(t)\,\mathrm dt\leq\ & (\vec \psi,\vec \varphi_k)^2_{L^2(\omsig)}\left(\frac{36\lambda_k }{\beta_{\min}(4\lambda_k-\beta_k^2)}+\frac{3 }{\sqrt{\lambda_0}} \right)\\
		& + 4(\alpha_k)^2\left(\frac{5\lambda_k}{\beta_{\min}}+\frac{\lambda_k\beta_{\max}}{4(4\lambda_k-\beta_k^2)}+ \frac{\beta_{\max}^2}{32\sqrt{\lambda_0}}\right)\\
		&+\left(\frac{2\lambda_k}{\beta_k(4\lambda_k-\beta_k^2)}+\frac{1}{8\sqrt{\lambda_0}}+\frac{4}{\beta_{\min}}\right)T\| f_k\|^2_{L^2(0,T)}.
	\end{align*}
	3. Let us consider the last case $\beta_k^2-4\lambda_k=0$. Then the solution to \eqref{eqn:ODE2} is given by
	\begin{align*}
		 {a_k} (t) &=\alpha_k(1+\frac{\beta_k}{2}t) e^{-\frac{\beta_k}{2}t}+(\vec \psi,\vec \varphi_k)_{L^2(\omsig)} te^{-\frac{\beta_k}{2}t} +\int\limits_0^t(t-s)e^{\frac{\beta_k}{2}(s-t)}f(s) \,\mathrm ds.
	\end{align*}
	With the estimate
	\[
	\int_0^t t^2 e^{-\beta_k t}\,\mathrm dt= \frac{1}{\beta_k^3} (2-e^{-\beta_kt}(\beta_k^2 t^2+2\beta_kt+2))\leq \frac{2}{\beta_k^3} 
	\]
	we then compute by using  $(a+b)^2\leq 2 a^2+2b^2$ and the Cauchy-Schwarz inequality that
	\begin{align*}
		\intT \lambda_k( {a_k})^2(t)\,\mathrm dt+\intT ( {a_k}')^2(t)\,\mathrm dt\leq\ & (\vec \psi,\vec \varphi_k)^2_{L^2(\omsig)}\left(\lambda_k\frac{8 }{\beta_k^3}+\frac{12}{\beta_k}\right)+4(\alpha_k)^2\left(\lambda_k\frac{3}{\beta_k}+\frac{1}{2\beta_k}\right)\\
		&+\left(\frac{2\lambda_k}{\beta_k^3}+\frac{3}{\beta_k}\right)2T\intT f_k^2(s)\,\mathrm ds.
	\end{align*}
	At last, we use  $4\lambda_k=\beta_k^2$ to derive	
	\begin{align*}
		\intT \lambda_k( {a_k})^2(t)\,\mathrm dt+\intT ( {a_k}')^2(t)\,\mathrm dt\leq\ & (\vec \psi,\vec \varphi_k)^2_{L^2(\omsig)}\frac{14 }{\beta_{\min}}+(\alpha_k)^2\lambda_k\frac{5}{4\beta_{\min}^3}	  +\frac{14}{\beta_{\min}^3}T\| f_k\|^2_{L^2(0,T)}.
	\end{align*}
	Adding all three cases will give the desired estimate.	
\end{proof}
Using the above lemma we can finally prove the last statement of Theorem \ref{thm:ExistenceAndUniqueness}.
\begin{prop}\label{prop:inequalDepRhs}
	Let the Assumption~\ref{A:Assumptions_all} hold true, $\vec j_a\in L^2(Q;\R^d)$, and $ \vec A$ be the unique solution of \eqref{vf:H1VarForm}. Then there exists positive constants  $c_\phi$, $c_\phi^c$, $c_\psi$, and $c_f$ such that the following inequality holds true
	\begin{align}
		\label{ineqn:ineq_LinDep}
		|\vec A|^2_{H^{\cu;1}(Q)}  \leq c_\phi\|\vec \phi\|^2_{L^2_{\varepsilon}({\Omega})} +c_\phi^cT\|\cu_x\vec \phi\|_{L^2_\mu({\Omega})}+c_\psi T\|\vec \psi\|^2_{L^2_{\varepsilon}({\Omega})}  + c_f \max\{T,T^2\} \|\vec j_a\|^2_{L^2_{\varepsilon}({Q})}.
	\end{align}
	The constants  $c_\phi$, $c_\phi^c$, $c_\psi$, and $c_f$  depend on $\sup\sigma$, $\sigma_{\min}^{-1}$, $\frac{1}{\sqrt{\lambda_{0}}}$ and $\max_{\substack{k\in \N_0\\\beta_k^2-4\lambda_k<0}}(4-\frac{\beta_k^2}{\lambda_k})^{-1}$ which is bounded since $\beta_k$ is bounded by $\sup\sigma$ and the non-zero eigenvalues $\lambda_k$ of $\cu_x\mu^{-1}\cu_x$ increase monotonically for $k\to\infty$. The $\beta_k\in\R_+$ are defined by $\sigma$ such that
	\begin{align*}
		(\sigma \vec \varphi_k,\vec \varphi_l)_{L^2(\omsig)} = \beta_k \delta_{kl}
	\end{align*}
	for the fundamental system $\{\vec \varphi_k\}$ of Lem.~\ref{lem:FS-H_0curl}.
\end{prop}
\begin{proof}
	Again we split the domain $\Omega$ into the support  $\omsig = \Omega \cap \mathrm{supp}(\sigma)$ of the conductivity $\sigma$ and its complement $\Omega_0:=\Omega\backslash\omsig$.
	First we take a look at $\omsig$ and $Q_\sigma := (0,T)\times\omsig$. To show the inequality we consider the basis representation \eqref{eq:BasisRepresentation} for the unique solution $\vec A$ of  \eqref{vf:H1VarForm}
	\begin{align*}
		\vec A(t,x)=\sum\limits_{k\in\Z} {a_k}(t)\vec \varphi_k(x),
	\end{align*}
	$(t,x) \in Q_\sigma$, where $\vec \varphi_k$ are the basis  functions of $H(\cu;\omsig)$ from Lemma \ref{lem:FS-H_0curl} .
	If we insert the representation into the energy norm , we get
	\begin{align}
		\nonumber
		(\varepsilon\partial_t \vec A,& \partial_t \vec A)_{L^2(Q_\sigma)}+(\mu^{-1}\cu_x \vec A,\cu_x \vec A)_{L^2(Q_\sigma)}\\
		\nonumber
		&\hspace*{.2cm}= \sum_{k,j\in\Z}  {( \partial_t {a_k},\partial_t c_j)_{(0,T)}(\varepsilon\vec \varphi_k,\vec \varphi_j)_{\omsig} + ({a_k}, c_j)_{(0,T)} (\mu^{-1}\cu_x\vec \varphi_k,\cu_x\vec \varphi_j)_{\omsig}}\\
		\label{eqn:energyNorm_linDep}
		&\hspace*{.2cm}=\sum_{k\in\N_0} {\norm{{a_k}'}_{L^2(0,T)}^2 +  \norm {\lambda_k^{1/2} {a_k}}_{L^2(0,T)}^2
			+\sum_{k\in \Z\backslash\N_0}\norm{{a_k}'}_{L^2(0,T)}^2}.
	\end{align}
	In the proof of Prop.~\ref{prop:Existence_H1_drhs} we saw that $ {a_k}$ satisfies one of two ordinary differential equations depending on the case of $k\in\N_0$ and $k\in\Z\backslash\N_0$. The first is
	\begin{align}
		\nonumber
		 {a_k}''(t)+ \beta_k  {a_k}'(t) &= f_k(t),\\
		\label{eqn:FirstDgl_LinDep}
		 {a_k}'(0) &= (\vec \psi,\vec \varphi_k)_{L^2(\omsig)},\\
		\nonumber
		 {a_k} (0) &=\alpha_k,
	\end{align}
	for  {$k\in \Z\backslash\N_0$} and the second is
	\begin{align}
		\nonumber
		 {a_k}''(t)+ \beta_k  {a_k}'(t)+ \lambda_k  {a_k}(t) &= f_k(t),\\
		\label{eqn:SecondDgl_LinDep}
		 {a_k}'(0) &= (\vec \psi,\vec \varphi_k)_{L^2(\omsig)},\\
		\nonumber
		 {a_k} (0) &=\alpha_k,
	\end{align}
	for $t\in (0,T)$ and $k\in \N_0$, where $f_k(t) =(\vec j_a,\vec \varphi_k)_{L^2(\omsig)} $ and $\lambda_k$ are the non-zero eigenvalues of the $\cu_x \mu^{-1} \cu_x$ operator.
	
	Let us first estimate $ {a_k}$ in the case that $k\in \Z\backslash\N_0$. Then we consider the first equation \eqref{eqn:FirstDgl_LinDep}, and compute the solution  
	\[ {a_k}(t)=\ \alpha_k+ \frac{1}{\beta_k}  (\vec \psi,\vec \varphi_k)_{L^2(\omsig)} (1 -e^{-\beta_k t}) + \frac{1}{\beta_k }\int\limits_0^t(1- e^{\beta_k (s-t)})f_k(s)\,\mathrm ds. \]	
	For this solution, we need to estimate the $L^2$-norm of $ {a_k}'$. Using $(a+b)^2 \leq 2a^2 + 2 b^2$, we arrive with the Cauchy-Schwarz inequality at
	\begin{align}
		\nonumber
		\intT ( {a_k}')^2\,\mathrm dt &\leq2(\vec \psi,\vec \varphi_k)_{L^2(\omsig)}^2\intT (  e^{-2\beta_k t})\,\mathrm dt +   2 \intT   \int\limits_0^t e^{2\beta_k (s-t)}\,\mathrm ds\int\limits_0^t f_k^2(s)\,\mathrm ds \,\mathrm dt\\
		\nonumber
		&\leq(\vec \psi,\vec \varphi_k)_{L^2(\omsig)}^2\frac{1- e^{-2\beta_k T}}{\beta_k} +   2 \frac{1-e^{-2\beta_k T}}{2\beta_k}T \|f_k\|_{L^2(0,T)}^2\\
		\label{ineqn:kinJ}
		&\leq (\vec \psi,\vec \varphi_k)_{L^2(\omsig)}^2\frac{1}{\beta_{\min}} +   \frac{1}{\beta_{\min}}T \|f_k\|_{L^2(0,T)}^2,
	\end{align}
	where $\beta_{\min} := \min_k \beta_k\geq\sigma_{\min} \varepsilon_{\max}^{-1}>0$.
	
	Second, we want to estimate $ {a_k}$ in the second case $k\in \N_0$. In this case, we can apply Lem.~\ref{lem:ODE2} for the second equation \eqref{eqn:SecondDgl_LinDep}. Hence we get the estimate 
	\begin{align}	
		\label{ineqn:kinI}
		\sum_{k\in \N_0} {\norm{{a_k}'}_{L^2(0,T)}^2 +  \norm {\lambda_k^{1/2} {a_k}}_{L^2(0,T)}^2}\leq \sum_{k\in \N_0} c_\psi (\vec \psi,\vec \varphi_k)^2_{L^2(\omsig)}+(c_\alpha^\lambda \lambda_k +c_\alpha ) (\alpha_k)^2 + c_f T\|f_k\|^2_{L^2(0,T)}.
	\end{align}
	Since $\frac{1}{\beta_{\min}}<c_\psi$ and $\frac{1}{\beta_{\min}}<c_f  $ we can combine (\ref{ineqn:kinI}) and  (\ref{ineqn:kinJ}) in \eqref{eqn:energyNorm_linDep} to arrive at
	\begin{align}
		\label{ineqn:kinN}
		\sum_{k\in\Z}  {\norm{{a_k}'}_{L^2(0,T)}^2 +  \norm {\lambda_k^{1/2} {a_k}}_{L^2(0,T)}^2  }\leq\sum_{k\in\N_0}&
		\left( c_\phi^1+ c_\phi^2\lambda_k\right)\alpha_k^2 +\sum_{k\in\Z}\big[c_\psi  (\vec \psi,\vec \varphi_k)_{L^2(\omsig)}^2 + c_fT\|f_k\|^2_{L^2(0,T)} \big].
	\end{align}
	As $\lambda_k$ were eigenvalues of the $\cu_x\mu^{-1}\cu_x$-opertor, we want to eliminate the $\lambda_k$ from our right hand side. 	Considering $\vec \phi = \sum_{k\in\Z} \alpha_k \vec \varphi_k$ and  $\cu_x \vec \phi = \sum_{k\in\N_0} \alpha_k \cu_x \vec \varphi_k$ we compute
	\begin{align*}
		\sum_{k=0}^\infty \lambda_k(\alpha_k)^2 &= \sum_{k=0}^\infty\sum_{l=0}^\infty \alpha_k\alpha_l \lambda_k(\varepsilon\vec \varphi_k,\vec \varphi_l)_{L^2(\omsig)}\\
		&=\sum_{k=0}^\infty\sum_{l=0}^\infty \alpha_k\alpha_l(\mu^{-1}\cu_x\vec \varphi_k,\cu_x\vec \varphi_l)_{L^2(\omsig)}\\
		&=(\mu^{-1}\cu_x\vec \phi,\cu_x\vec \phi)_{L^2(\omsig)}.
	\end{align*}
	Using the basis representation for $\vec \phi$ and $\vec j_a$, where $f_k =(\vec j_a,\vec \varphi_k)_{L^2(\omsig)}$, respectively in (\ref{ineqn:kinN}) we then arrive at
	\begin{align*}
		(\varepsilon\partial_t \vec A, \partial_t \vec A)_{L^2(Q_\sigma)}+&(\mu^{-1}\cu_x  \vec A,\cu_x  \vec A)_{L^2(Q_\sigma)}\\
		\leq&\, c_\psi\|\vec \psi\|^2_{L^2_{\varepsilon}({\omsig})}  +c_\phi\|\vec \phi\|^2_{L^2_{\varepsilon}({\omsig})} + c_\phi^c\|\cu_x\vec \phi\|^2_{L^2_{\mu}({\omsig})} + c_f T \|\vec j_a\|^2_{L^2_{\varepsilon}({Q_\sigma})}.
	\end{align*}
	
	For the solution over $\Omega_0 = \Omega\backslash\overline{\omsig}$ and $Q_0:=(0,T)\times\Omega_0 $, where $\sigma \equiv 0$,  we get the ordinary equation  given by  \eqref{eqn:ode2_sigma_equal_0} for $k \in \Z$.  The equation \eqref{eqn:ode2_sigma_equal_0} for $\lambda_k =0$ yields the estimate  
		\[\norm{ {a_k}'}_{L^2(0,T)}^2 \leq T^2 \|f_k\|^2_{L^2(0,T)}. \]
		On the other hand, if $\lambda_k >0$, then the equation  \eqref{eqn:ode2_sigma_equal_0} has the solution
	\[
	 {a_k}(t) = \frac{1}{\sqrt{\lambda_k}}(\vec \psi,\vec \varphi_k)_{L^2(\Omega_0)}\sin(\sqrt{\lambda_k}t) + \alpha_k\cos(\sqrt{\lambda_k}t) + \frac{1}{\sqrt{\lambda_k}} {(f_k(.),\sin(\sqrt{\lambda_k}(t-.)))_{L^2(0,t)}}.
	\]
	To estimate the $L^2(0,T)$ norm of $ a_k$ and $ a_k'$, we follow the same steps as above and derive
	\begin{align}
		\label{ieqn:outhatOmega}
		(\varepsilon\partial_t \vec A, \partial_t \vec A)_{L^2(Q_0)}+&(\mu^{-1}\cu_x  \vec A,\cu_x  \vec A)_{L^2(Q_0)} \\
		\nonumber
		\leq&	4T\|\vec \psi\|_{L^2_{\varepsilon}(\Omega_0)}^2 +4T\|\cu_x \phi\|^2_{L^2_{\mu}(\Omega_0)} +  T^2\|j_a\|^2_{L^2_\varepsilon(Q_0)}.
	\end{align}
	Note that we get a $T$ in our estimates by estimating $\intT \sin^2(\sqrt{\lambda_k}t)\,\mathrm dt \leq T$ and the same integral with cosines.
	Adding the inequality (\ref{ieqn:outhatOmega}) brings us to the desired inequality (\ref{ineqn:ineq_LinDep})
	\begin{align*}
		| \vec A |_{H^{\cu;1}(Q)}	\leq \tilde{c}_\psi T\|\vec \psi\|^2_{L^2_{\varepsilon}({\Omega})}  +c_\phi\|\vec \phi\|^2_{L^2_{\varepsilon}({\Omega})} +\tilde{c}_\phi^cT \|\cu_x\vec \phi\|^2_{L^2_{\mu}({\Omega})} + c_f \max\{T,T^2\} \|\vec j_a\|^2_{L^2_{\varepsilon}({Q})}.
	\end{align*}
\end{proof}

\begin{rmk}
	Note that with the tools of Ladyzhenskaya \cite[Thm.~3.2, p.~160]{Ladyzhenskaya1985} one can translate the results of the scalar wave equation to the vectorial wave equation using the same technique as the above theorem, see \cite{HauserDiss2021} for more details. Then we would derive the estimate
	\begin{align}
		\label{eqn:ladyschenskaya_absch}
		(\varepsilon\partial_t \vec A, \partial_t \vec A)_{L^2(Q)}+(\mu^{-1}\cu_x &\vec A,\cu_x  \vec A)_{L^2(Q)} \\
		\nonumber
		\leq&	4T\|\vec \psi\|_{[L^2(\Omega)]^d}^2 +4T\|\cu_x \vec \phi\|^2_{L^2_{\mu}(\Omega)} +  T^2\|\vec j_a\|^2_{[L^2(Q)]^d}.
	\end{align}
	for the solution of the variational formulation:\\
	Find $ \vec A\in H^{\cu;1}_{0;} (Q) $ with $\vec A(0,.) = \vec \phi$ such that
	\begin{align*}
		-(\varepsilon\partial_t \vec A,\partial_t \vec v)_{L^2(Q)}+ (\mu^{-1}	\cu_x  \vec A,\cu_x \vec v)_{L^2(Q)} &= (\vec j_a, \vec v)_{L^2(Q)} -(\varepsilon	\vec \psi,\vec v(0,\cdot))_{L^2(\Omega)}
	\end{align*}
	for all $\vec v\in H^{\cu;1}_{0;,0}(Q) $. 
\end{rmk}

\section{A space-time finite element method in a tensor product}
In Theorem \ref{thm:ExistenceAndUniqueness} we have proven the unique solvability of the variational formulation \eqref{vf:H1VarForm}. Now we take a look at the discrete equivalent of the variational formulation for the vectorial wave equation. First, we define the finite element spaces that we will be using. Then we derive the finite element discretization and analyze the resulting linear equation. Afterward, we derive a CFL condition  that indeed seems to be necessary for stable computations of the numerical examples.
\subsection{Space-time finite element spaces}\label{sec:FES}
Before we compute any examples, we recap the main ideas of the space-time discretization of the vectorial wave equation as was presented in \cite{HauserZank2023}. We will use similar nomenclature to ease the comparison of the results of both papers. The main difference to \cite{HauserZank2023} is the addition of the conductivity $\sigma$ in our equation and that we consider second-order elements in time instead of first-order elements as in \cite{HauserZank2023}.

For the discretization in time we decompose the time interval $(0,T)$ with $N_t^2+1$ points
\begin{equation*}
	0 = t_0 < t_1 < \dots < t_{N_t^2-1} < t_{N_t^2} = T
\end{equation*}
into $N_t^2$ subintervals  $\tau_\ell = (t_{\ell-1}, t_\ell) \subset \R$, $\ell=1,\dots, N_t^2.$ We define the local mesh size of each element $\tau_l$ as $h_{t,\ell} = t_\ell - t_{\ell-1}$, $\ell=1,\dots, N_t^2,$ with the maximal mesh size $h_t := \max_{\ell=1,\dots, N_t^2} h_{t,\ell}.$ 

For the discretization in space we decompose the spatial domain $\Omega \subset \R^d$ into $N^x$ elements $\omega_i \subset \R^d$ for $i=1,\dots, N^x$, satisfying
\begin{equation*}
	\overline{\Omega} = \bigcup_{i=1}^{N^x} \overline{\omega_{i}}\qquad\text{and}\qquad \omega_{i}\cap \omega_{j} = 0 \qquad\text{for } i\neq j.
\end{equation*}
 We define  the local mesh size in $\omega_i$ as  $h_{x,i} := \left( \int_{\omega_i} 1 \mathrm dx \right)^{1/d}$, $i=1,\dots, N^x$, the maximal mesh size as $h_x := h_{x,\max}(\mathcal T^x_\nu) := \max_{i=1,\dots, N^x} h_{x,i}$ and the minimal mesh size as $h_{x,\min}(\mathcal T^x_\nu) := \min_{i=1,\dots, N^x} h_{x,i}$. The parameter $\nu \in \N$ is the level of refinement and the spatial mesh is the set $\mathcal T^x := \mathcal T^x_{\nu} = \{ \omega_i \}_{i=1}^{N^x}$ at level $\nu$. For simplicity, we only consider triangles for $d=2$ and tetrahedra for $d=3$ as the spatial elements $\omega_i \subset \R^d$, $i=1,\dots, N^x$.

\begin{assumption} \label{A:Assumptions_mesh}
	We assume that the sequence $( \mathcal T^x_\nu )_{\nu \in \N}$ consists of shape-regular, globally quasi-uniform and admissible spatial meshes $\mathcal T^x_\nu$. 
\end{assumption}

 By the shape-regularity of the mesh sequence $( \mathcal T^x_\nu )_{\nu \in \N}$  there is a constant $c_\mathrm F > 0$ such that
\begin{equation} \label{FES:formregular}
	\forall \nu \in \N: \forall \omega \in \mathcal T^x_\nu: \quad \sup_{x,y \in \overline{\omega} } \norm{x - y}_2 \leq c_\mathrm F\, r_\omega,
\end{equation}
where  $\norm{\cdot}_2$ is the Euclidean norm in $\R^d$, and $r_\omega >0$ is the radius of the largest ball that can be inscribed in the element $\omega$.  By the global quasi-uniformity of the mesh sequence $(\mathcal T^x_\nu)_{\nu \in \N}$   there is a constant $c_\mathrm G \geq 1$ such that
\begin{equation*}
	\forall \nu \in \N\colon \quad \frac{h_{x,\max}(\mathcal T^x_\nu)}{h_{x,\min}(\mathcal T^x_\nu)} \leq c_\mathrm G.
\end{equation*}
 On the other hand, a mesh is admissible if for each two elements $\tau_i,\tau_j\in \mathcal T^x_\nu$ the intersection $S = \tau_i\cap \tau_j$ is either empty or a simplex of the same type belonging to both elements.

With the discretization of the time interval and spatial domain in mind, we now define finite element spaces. For the test and ansatz space in time, we define the space of piecewise quadratic, continuous functions
\begin{equation*}
	S^2(\mathcal T^t) := \Big\{ v_{h_t} \in C[0,T]\, |\, \forall \ell \in \{ 1, \dots, N_t^2 \}: v_{h_t|\overline{\tau_\ell}} \in P^2(\overline{\tau_\ell})  \Big\}  = \mathrm{span} \{\varphi^2_\ell\}_{\ell=0}^{N_t^2}
\end{equation*}
with the usual nodal basis functions $\varphi^2_\ell$, satisfying $\varphi^2_\ell(t_\kappa)=\delta_{\ell \kappa}$ for $\ell,\kappa=0,\dots, N_t^2$. The space $P^p(B)$ is the space of polynomials on a subset $B \subset \R$ of global degree at most $p \in \N_0$, and $\delta_{\ell \kappa}$ is the Kronecker delta. Then, we introduce the subspaces which include initial and end conditions
\begin{align*}
	S_{0,}^2(\mathcal T^t) &:= S^2(\mathcal T^t) \cap H^1_{0,}(0,T) = \mathrm{span} \{\varphi^2_\ell\}_{\ell=1}^{N_t^2}, \\
	S_{,0}^2(\mathcal T^t) &:= S^2(\mathcal T^t) \cap H^1_{,0}(0,T) = \mathrm{span} \{\varphi^2_\ell\}_{\ell=0}^{N_t^2-1}.
\end{align*}

In space, we introduce the vector-valued finite element spaces of Nédélec and Raviart--Thomas finite elements. For a spatial element $\omega \in \mathcal T^x_\nu$, we define the polynomial spaces
\begin{equation*}
	\mathcal{RT}^0(\omega) := \Big\{ \vec v \in \mathbb [P^1(\omega)]^d\,|\, \forall \vec x \in \omega: v(\vec x) = \vec a + b \vec x \text{ with } \vec a \in \R^d, b \in \R \Big\}
\end{equation*}
and
\begin{equation*}
	\mathcal N_\mathrm{I}^0(\omega) := \Big\{ \vec v \in \mathbb [P^1(\omega)]^2\,|\, \forall (x_1,x_2) \in \omega: v(x_1,x_2) = \vec a + b \cdot (-x_2,x_1)^\top \text{ with } \vec a \in \R^2, b \in \R \Big\}
\end{equation*}
for $d=2$, and
\begin{equation*}
	\mathcal N_\mathrm{I}^0(\omega) := \Big\{ \vec v \in \mathbb [P^1(\omega)]^3\,|\, \forall \vec x \in \omega: v(\vec x) = \vec a + \vec b \times \vec x
	\text{ with } \vec a \in \R^3, \vec b \in \R^3 \Big\}
\end{equation*}
for $d=3$, see \cite[Section~14.1, Section~15.1]{ErnGuermond2020I} for a more thorough introduction. With this notation, we define the space of the lowest-order Raviart--Thomas finite element space and the lowest-order Nédélec finite element space of the first kind
\begin{align*}
	\mathcal{RT}^0(\mathcal T^x_\nu) &:= \left\{ \vec v_{h_x} \in H(\mathrm{div};\Omega) \,| \, \forall \omega \in \mathcal T^x_\nu : \vec v_{h_x| \omega} \in \mathcal{RT}^0(\omega) \right\} = \mathrm{span} \{\vec \psi_k^{\mathcal{RT}} \}_{k=1}^{N_x^{\mathcal{RT}}},\\
	\mathcal N_\mathrm{I}^0(\mathcal T^x_\nu) &:= \left\{ \vec v_{h_x} \in H(\cu;\Omega) \,|\, \, \forall \omega \in \mathcal T^x_\nu : \vec v_{h_x| \omega} \in \mathcal N_\mathrm{I}^0(\omega) \right\},
\end{align*}
see  \cite{ErnGuermond2020I,Monk} for more details on these spaces. Further, we consider the subspace of $\mathcal N_\mathrm{I}^0(\mathcal T^x_\nu)$ with zero  tangential trace
\begin{equation*}
	\mathcal N_\mathrm{I,0}^0(\mathcal T^x_\nu) := \mathcal N_\mathrm{I}^0(\mathcal T^x_\nu) \cap H_0(\cu;\Omega) = \mathrm{span} \{\vec \psi_k^{\mathcal{N}} \}_{k=1}^{N_x^{\mathcal{N}}}.
\end{equation*}
 {Let us consider examples of shape functions for both $\mathcal N_\mathrm{I}^0(\mathcal T^x_\nu)$ which associates with the edges and $\mathcal{RT}^0(\mathcal T^x_\nu)$ which associates with the the faces. For an edge $e_{ij}$ going from point $p_i$ to $p_j$ the corresponding shape function would look like $\psi^{\mathcal{N}}_{e_{ij}} = \lambda_i\nabla \lambda_j+\lambda_j\nabla \lambda_i)$, where $\lambda_i$ and $\lambda_j$ are the corresponding barycentric function to $p_i$ and $p_j$. On the other hand, an example for the lowest order shape functions associated to a face $F_{ijk}$ consisting of the three points $p_i$, $p_j$ and $p_k$ would look like $\psi_{F_{ijk}}^{\mathcal{RT}}=\lambda_i\nabla \lambda_j\times \nabla \lambda_k +\lambda_j\nabla \lambda_k\times \nabla \lambda_i +\lambda_k\nabla \lambda_i\times \nabla \lambda_j $, see \cite[Section~14.1, Section~15.1]{ErnGuermond2020I} for more details.}

Last, the temporal and spatial meshes $\mathcal T^t = \{ \tau_\ell \}_{\ell=1}^{N_t^2}$ and $\mathcal T^x_\nu = \{ \omega_i \}_{i=1}^{N^x}$ lead to a decomposition
\begin{equation*}
	\overline{Q} = [0,T] \times \overline{\Omega} = \bigcup_{\ell=1}^{N_t^2} \overline{\tau_\ell} \times \bigcup_{i=1}^{N^x} \overline{\omega_i}
\end{equation*}
of the space-time cylinder $Q \subset \R^{d+1}$ with $N_t^2\cdot N^x$ space-time elements. Therefore $\mathcal T^t  \times \mathcal T^x_\nu$ is a space-time tensor product mesh. To this space-time mesh, we relate space-time finite element spaces of tensor-product type using $\hat{\tens}$ as the Hilbert tensor-product. These spaces are
\begin{equation} \label{FES:TPRaueme}
	S_{0,}^2(\mathcal T^t) \hat{\tens}\mathcal N_\mathrm{I,0}^0(\mathcal T^x_\nu) \quad \text{ and } \quad S_{,0}^2(\mathcal T^t) \hat{\tens}\mathcal N_\mathrm{I,0}^0(\mathcal T^x_\nu).
\end{equation}
Then, any function $\vec A_h \in S_{0,}^2(\mathcal T^t) \hat{\tens}\mathcal N_\mathrm{I,0}^0(\mathcal T^x_\nu)$ admits the representation
\begin{equation} \label{eqn:DefAh}
	\vec A_h(t,x) = \sum_{\kappa=1}^{N_t^2} \sum_{k=1}^{N_x^\mathcal{N}} \mathcal A_k^\kappa \varphi^2_\kappa(t) \vec \psi^\mathcal{N}_k(x)
\end{equation}
for $ (t,x) \in \overline{Q}$ with coefficients $\mathcal A_k^\kappa \in \R$. Further, for a given function $\vec f \in L^2(Q; \R^d)$, we introduce the $L^2(Q)$ projection  $\Pi_h^{\mathcal{RT},1} \colon \, L^2(Q; \R^d) \to S^1(\mathcal T^t) \hat{\tens} \mathcal{RT}^0(\mathcal T^x_\nu)$ to find $\Pi_h^{\mathcal{RT},1} \vec f \in S^1(\mathcal T^t) \hat{\tens} \mathcal{RT}^0(\mathcal T^x_\nu)$ such that
\begin{equation} \label{FES:L2RT}
	(\Pi_h^{\mathcal{RT},1} \vec f, \vec w_h)_{L^2(Q)} = (\vec f, \vec w_h)_{L^2(Q)}.
\end{equation}
for all $ \vec w_h \in S^1(\mathcal T^t) \hat{\tens} \mathcal{RT}^0(\mathcal T^x_\nu)$. Since we need enough regularity of the right-hand side such that the continuity equation \eqref{eqn:ContinuityEq} is included in the vectorial wave equation, we require that the divergence of the right-hand side $j_a$ is well defined. Therefore we are interested in the projection $\Pi_h^{\mathcal{RT},1} $. On the other side, we see in \cite{HauserZank2023} what kind of effect another projection can have, namely  the production of so-called `spurious modes'. 

With these definitions in mind, we can now formulate a conforming finite element approach for the variational formulation \eqref{vf:H1VarForm}.  We use $S^2(\mathcal T^t) \hat{\tens}\mathcal N_\mathrm{I,0}^0(\mathcal T^x_\nu)$ as trial space  and $S_{,0}^2(\mathcal T^t) \hat{\tens}\mathcal N_\mathrm{I,0}^0(\mathcal T^x_\nu)$ as test space. Then we end up with the discrete variational formulation:\\
Find $ \vec A_h\in  S^2(\mathcal T^t) \hat{\tens}\mathcal N_\mathrm{I,0}^0(\mathcal T^x_\nu)$, with $\vec A_h(0,x)=\vec \phi$, such that
\begin{align}\label{vf:H1_disc}
	-\left(\varepsilon\partial_t \vec A_h, \partial_t \vec v_h \right)_{L^2(Q)}+\left(\sigma\partial_t \vec A_h, \vec v_h \right)_{L^2(Q)}  +&\left(\mu^{-1}	\mathrm{curl}_x \vec A_h, \mathrm{curl}_x \vec v_h \right)_{L^2(Q)}\\
	\nonumber
	& = \left(\Pi_h^{\mathcal{RT},1} \vec j_a, \vec v_h \right)_{L^2(Q)} - \left(\varepsilon\vec \psi, \vec v_h (0,.)\right)_{L^2(\Omega)}
\end{align}
for all $\vec v_h \in  S_{,0}^2(\mathcal T^t) \hat{\tens}\mathcal N_\mathrm{I,0}^0(\mathcal T^x_\nu)$.

\subsection{The finite element discretization}
Next, we follow the same steps as \cite{HauserZank2023} to derive the finite element discretization, but include additional terms involving the conductivity $\sigma$.  Let, for a moment, the initial condition $\vec\phi$ be zero. Then we insert \eqref {eqn:DefAh} into the  discrete variational formulation (\ref{vf:H1_disc}). Afterward, the integrals split into temporal and spatial matrices.
Hence, the discrete variational formulation~\eqref{vf:H1_disc} is equivalent to the linear system
\begin{equation} \label{FEM:LGS}
	(-A_{tt} \otimes M_{x} + A_{t} \otimes M_{x}^{\sigma}  + M_{t} \otimes A_{xx}) \vec{ \mathcal A }= \vec{\mathcal J}
\end{equation}
with the temporal matrices
\begin{subequations}  \label{FEM:Ohne:Zeitmatrizen}
	\begin{equation}
		A_{tt}[\ell,\kappa] = (\partial_t \varphi^2_\kappa, \partial_t \varphi^2_\ell)_{L^2(0,T)}, \quad M_{t}[\ell,\kappa] = (\varphi^2_\kappa, \mathcal  \varphi^2_\ell)_{L^2(0,T)},
	\end{equation}
	\begin{equation}
		A_{t}[\ell,\kappa] = (\partial_t \varphi^2_\kappa,  \varphi^2_\ell)_{L^2(0,T)}
	\end{equation}
\end{subequations}
for  $\ell=0,\dots,N_t^2-1$, $\kappa=1,\dots,N_t^2$ and the spatial matrices
\begin{subequations}
	\label{FEM:Ortsmatrizen}
	\begin{equation} \label{FEM:Ortsmatrizen1}
		A_{xx}[l,k] = (\mu^{-1}\cu_x \vec \psi^\mathcal{N}_k, \cu_x \vec \psi^\mathcal{N}_l)_{L^2(\Omega)}, \quad M_{x}[l,k] = (\varepsilon \vec \psi^\mathcal{N}_k, \vec \psi^\mathcal{N}_l)_{L^2(\Omega)},
	\end{equation}
	\begin{equation} \label{FEM:Ortsmatrizen2}
		M_{x}^{\sigma}[l,k] = (\sigma \vec \psi^\mathcal{N}_k, \vec \psi^\mathcal{N}_l)_{L^2(\Omega)}
	\end{equation}
\end{subequations}
for $k,l=1,\dots,N_x^\mathcal{N}.$ The degrees of freedom are ordered such that the vector of coefficients in  \eqref{eqn:DefAh} reads as
\begin{equation} \label{FEM:Ohne:VectorAh}
	\vec{A} = ( \vec{A}^1, \vec{A}^2, \dots, \vec{A}^{N_t^2} )^\top \in \R^{N_t^2 N_x^\mathcal{N}}
\end{equation}
with
\begin{equation*}
	\vec{A}^\kappa = (\mathcal A_1^\kappa, \mathcal A_2^\kappa, \dots, \mathcal A_{N_x^\mathcal{N}}^\kappa)^\top \in \R^{N_x^\mathcal{N}} \quad \text{ for } \kappa \in \{1,\dots, N_t^2\}.
\end{equation*}
The right-hand side in \eqref{FEM:LGS} is given in the same way by
\begin{equation*}
	\vec{\mathcal J} = ( \vec f^0, \vec f^1, \dots, \vec f^{N_t^2-1} )^\top \in \R^{N_t^2 N_x^\mathcal{N}}, 
\end{equation*}
where we define
\begin{equation*}
	\vec f^\ell = (f_1^\ell, f_2^\ell, \dots, f_{N_x^\mathcal{N}}^\ell)^\top \in \R^{N_x^\mathcal{N}} 
\end{equation*}
with
\begin{equation} \label{FEM:Ohne:RSf}
	f_l^\ell = (\Pi_h^{\mathcal{RT},1} \vec j,\varphi^2_\ell \vec \psi^\mathcal{N}_l )_{L^2(Q)} -\varphi^2_\ell(0)  \left(\varepsilon\vec \psi, \vec \psi^\mathcal{N}_l\right)_{L^2(\Omega)} 
\end{equation}
for  $\ell=0,\dots, N_t^2-1$ and   $l=1,\dots,N_x^\mathcal{N}.$
\paragraph{Remark}
If we have non-zero initial condition $\vec \phi$, then  we consider a continuous extension in time $\tilde{\phi}$ that is zero on all other degrees of freedom in time. Hence, for our application we use the representation (\ref{eqn:DefAh}) to derive for $t=0$
\begin{align*}
	\tilde{\phi}_{h}(0,x) = \sum_{k = 1}^{N_x^{\mathcal{N}}} A_k^0 \vec{\psi}^{\mathcal{N}}_k(x)  \approx \tilde{\phi}(0,x) .
\end{align*}
Then, using the ordering of the degrees of freedom from Section~\ref{sec:FES}, we end up with the system
\begin{align}
	\label{eqn:sys_space_time_tensor}
	(-A_{tt} \otimes M_{x} + A_{t} \otimes M_{x}^{\sigma}  + M_{t} \otimes A_{xx}) \vec{ \mathcal A } =  \vec{\mathcal J}-\vec{A}_{Ini},
\end{align}
where
\begin{align*}
	\vec A_{Ini} &\coloneqq (\vec A_{Ini}^0,\dots, \vec A_{Ini}^{N_t^2-1}),\quad 
	\vec A_{Ini}^{\ell} \coloneqq (A_{Ini;1}^{\ell},\dots, A_{Ini;N_x^{\mathcal{N}}}^{\ell}),\\
	A_{Ini;l}^{\ell}&\coloneqq \sum_{k=1}^{N_x^{\mathcal{N}}} (-A_{tt}[\ell,0] M_{x}[l,k] + A_{t}[\ell,0]  M_{x}^{\sigma}[l,k]  + M_{t} [\ell,0] A_{xx}[l,k])A_k^0 
\end{align*}
for $\ell = 0,\dots, N_t^2-1 $ and $l = 1,\dots, N_x^{\mathcal{N}} $ with the matrices defined in \eqref{FEM:Ortsmatrizen}, \eqref{FEM:Ohne:Zeitmatrizen} and
\begin{equation*}
	A_{tt}[\ell,0]  \coloneqq (\partial_t \varphi^2_0, \partial_t \varphi^2_\ell)_{L^2(0,T)},\quad	
	M_{t} [\ell,0]  \coloneqq  (\varphi^2_0, \mathcal  \varphi^2_\ell)_{L^2(0,T)},
\end{equation*}
\begin{equation*}
	A_{t}[\ell,0]  	\coloneqq (\partial_t \varphi^2_0,  \varphi^2_\ell)_{L^2(0,T)}.
\end{equation*}

\subsection{The CFL condition}\label{sec:cfl-condition}
If we solve the equation \eqref{FEM:LGS} for simple examples, we see conditional stability in our results which hints at a CFL condition. To compute this CFL condition for functions with second-order elements in time we use a similar tactic as \cite{SteinbachZankETNA2020} did for linear elements. To simplify the calculation we assume that we have an equidistant discretization in time with step size $h_t$. 

First let $\sigma = 0$. Note that, since $\sigma$ acts like a stabilizer to our system, the  case  $\sigma \equiv 0$ is the hardest case. Let $\lambda_k$, $k\in \N_0$, be the eigenvalues of the $\text{curl} \mu^{-1} \text{curl}$ operator in the weighted $L^2_\varepsilon(\Omega)$-norm from Lemma~\ref{lem:FS-H_0curl}. To analyze the stability of the discrete system we choose initial data $\vec \phi=0$, $\vec \psi =0$ as well as the right-hand side $\vec j_a=0$. Now, we consider the basis ansatz from Section~\ref{sec:BasisRep} for $\vec A\in H^{\textmd{curl};1}_{0;0,}(Q)$ and write
\begin{align*}
	\vec A^N = \sum\limits_{k=-N}^N \tilde{A}_k(t)\vec \varphi_k(x)
\end{align*}

Then we can rewrite the variational formulation \eqref{vf:H1VarForm} into a variational formulation in the time domain only by using the properties of the eigenvalues $\lambda_i$ of the spatial operator $\text{curl} \mu^{-1} \text{curl}$. We end up with the set of variational formulations: Find $\tilde{A}_k \in H^1_{0,}(0,T) $ such that
\begin{align}\label{vf:TimeVF}
	-	(\partial_t \tilde{A}_k,\partial_t  v)_{(0,T)}  +\lambda_k	( \tilde{A}_k,  v)_{(0,T)}  = 0
\end{align}
for all  $ v\in  H^1_{,0}(0,T)  $ and all non-zero eigenvalues $\lambda_k>0$ where $k\in\N_0$. On the other hand, for $k\in \Z\backslash \N_0$, we end up with the formulation: Find $\tilde{A}_k \in H^1_{0,}(0,T) $ such that
\begin{align*}
	(\partial_t \tilde{A}_k,\partial_t  v)_{(0,T)} = 0
\end{align*}
for all  $ v\in  H^1_{,0}(0,T)$. This is equivalent to the Laplace equation and does not add to the numerical instabilities that can be observed in section \ref{sec:results}. 

To derive the appropriate CFL condition, we need to analyze the stability of \eqref{vf:TimeVF} in the discrete  case.  Using again the ansatz and test spaces $S_{0,}^2(\mathcal T^t)$  and $S_{,0}^2(\mathcal T^t)$ we get the discrete formulation:\\
Find $\tilde{A}_k^h \in S_{0,}^2(\mathcal T^t) $ such that
\begin{align*}
	-	(\partial_t \tilde{A}_k^h,\partial_t  v_h )_{(0,T)}  +\lambda_k	( \tilde{A}_k^h,  v_h )_{(0,T)}  = 0
\end{align*}
for all  $ v_h \in S_{,0}^2(\mathcal T^t)  $ and all non-zero eigenvalues $\lambda_k$, $k\in\N_0$.  This is equivalent to analyzing the linear system
\begin{align}\label{eq:TimeDiscrete}
	(-	A_{tt} +\lambda_k M_t)  \vec{\tilde{A}} = 0
\end{align}
for its stability.  The temporal matrices $A_{tt}$ and $M_t$ are defined in \eqref{FEM:Ohne:Zeitmatrizen} and the solution $ \vec{\tilde{A}}=(\tilde{A}_k(t_1),\dots,\tilde{A}_k(T))^T\in \R^{N_t^2}$ is the coefficient vector  of the basis representation $\tilde{A}_k^h(t) = \sum_{\kappa=1}^{N_t^2}\tilde{A}_k(t_\kappa)\varphi_\kappa^2(t).$ For a moment let us write $\lambda = \lambda_k$ and $h=h_t$. Then we compute the element matrices
\begin{align*}
	M^e_t &= \frac{h}{30} \begin{pmatrix}
		4&2&-1\\2&16&2\\-1&2&4
	\end{pmatrix}, &A_{tt}^e = \frac{1}{3h} \begin{pmatrix}
		7&-8&1\\-8&16&-8\\1&-8&7
	\end{pmatrix}.
\end{align*}
By using the element matrices to assemble the system matrix we compute
\begin{align*}
	K_{\text{sys}} = \begin{pmatrix}
		\frac{8}{3h}+\frac{\lambda h}{15}& -\frac{1}{3h}-\frac{\lambda h}{30}&&&&&&\\
		\frac{-16}{3h}+\frac{8\lambda h}{15}& \frac{8}{3h}+\frac{\lambda h}{15}&&&&&&\\
		\frac{8}{3h}+\frac{\lambda h}{15}& -\frac{14}{3h}+\frac{4\lambda h}{15}&\frac{8}{3h}+\frac{\lambda h}{15}& -\frac{1}{3h}-\frac{\lambda h}{30}&&&&\\
		&\frac{8}{3h}+\frac{\lambda h}{15}&\frac{-16}{3h}+\frac{8\lambda h}{15}& \frac{8}{3h}+\frac{\lambda h}{15}&&&&\\
		&-\frac{1}{3h}-\frac{\lambda h}{30}&\frac{8}{3h}+\frac{\lambda h}{15}& -\frac{14}{3h}+\frac{4\lambda h}{15}&\frac{8}{3h}+\frac{\lambda h}{15}& -\frac{1}{3h}+\frac{\lambda h}{30}&\dots&\\
		&&&\dots&\dots&\dots&\dots&\\
		&&&&\frac{8}{3h}+\frac{\lambda h}{15}&\frac{-16}{3h}+\frac{8\lambda h}{15}& \frac{8}{3h}+\frac{\lambda h}{15}
	\end{pmatrix}.
\end{align*}
By multiplying with $3h$ we get the following recursive formula
\begin{align*}
	\begin{pmatrix}
		{8}+\frac{\lambda h^2}{5}& -{1}-\frac{\lambda h^2}{10}\\
		{-16}+\frac{8\lambda h^2}{5}& {8}+\frac{\lambda h^2}{5}\\
	\end{pmatrix}
	\begin{pmatrix}
		u_{2\kappa-1}\\u_{2\kappa} 
	\end{pmatrix}
	=
	\begin{pmatrix}
		-{8}-\frac{\lambda h^2}{5}& {14}-\frac{4\lambda h^2}{5}\\
		0& -{8}-\frac{\lambda h^2}{5}\\
	\end{pmatrix}
	\begin{pmatrix}
		u_{2\kappa-3}\\u_{2\kappa-2} 
	\end{pmatrix}
	+
	\begin{pmatrix}
		0& {1}+\frac{\lambda h^3}{10}\\
		0&0\\
	\end{pmatrix}
	\begin{pmatrix}
		u_{2\kappa-5}\\u_{2\kappa-4} 
	\end{pmatrix}
\end{align*}
for $\kappa=2,\dots, N_t^2/2$. Note, that $N_t^2/2\in\N$.
This recursive formula can be understood as a two-step method 
\begin{equation*}
	A\vec z_\kappa =
	B_1 z_{\kappa-1}+ B_2 z_{\kappa-2}
\end{equation*}
for the vector $\vec z_\kappa = (u_{2\kappa-1},u_{2\kappa} )^T$ with 
\begin{subequations}
	\begin{equation*}
		A :=\begin{pmatrix}
			{8}+\frac{\lambda h^2}{5}& -{1}-\frac{\lambda h^2}{10}\\
			{-16}+\frac{8\lambda h^2}{5}& {8}+\frac{\lambda h^2}{5}\\
		\end{pmatrix}, 
	\end{equation*}
	\begin{equation*}	
		B_1 := \begin{pmatrix}
			-{8}-\frac{\lambda h^2}{5}& {14}-\frac{4\lambda h^2}{5}\\
			0& -{8}-\frac{\lambda h^2}{5}\\
		\end{pmatrix},\qquad 
		B_2 \, := \begin{pmatrix}
			0& {1}+\frac{\lambda h^3}{10}\\
			0&0\\
		\end{pmatrix}.
	\end{equation*}
\end{subequations} 
Further, we can rewrite the two-step method as the system
\begin{equation}\label{eq:TimeSystem}
	Y_\kappa = A_{\text{sys}} Y_{\kappa-1}
\end{equation}
with $Y_\kappa =(u_{2\kappa-3},u_{2\kappa-2},u_{2\kappa-1}, u_{2\kappa})^T$ and 
\begin{equation*}
	A_{\text{sys}} := \begin{pmatrix}
		0&I\\A^{-1}B_2&A^{-1} B_1
	\end{pmatrix}
\end{equation*}
for $\kappa=2,\dots, N_t^2/2$. Next, we solve \eqref{eq:TimeSystem} by iterate in $\kappa$ and arrive at the formula
\begin{align*}
	Y_\kappa = A_{\text{sys}}^{\kappa-1}Y_{1}
\end{align*}
for the solution $Y_\kappa$ with the initial values 
\begin{align*}
	\begin{pmatrix}
		u_1\\u_2
	\end{pmatrix}
	= A^{-1}\begin{pmatrix}
		7-\frac{2\lambda h^2}{5}\\
		-{8}-\frac{\lambda h^2}{5}
	\end{pmatrix}
	u_0,
\end{align*}
where $u_{-1} = 0$ and $u_0= \tilde{A}_k(0)$. Therefore, if we want our system to be zero-stable, we need the real parts of the eigenvalues of $A_{\text{sys}}$  to be less than one. By analyzing the eigenvalues of $A_{\text{sys}}$ we see that two of the eigenvalues $\lambda^A_1$ and $\lambda^A_2$ are zero. The other two are given by the formula
\begin{align*}
	\lambda^A_{3,4} = \frac{-2a^2-bd\pm\sqrt{b(2c+d)(4a^2-2bc+bd)}}{2(a^2-bc)}
\end{align*} 
where
\begin{align*}
	a &= {8}+\frac{\lambda h^2}{5}, &b ={-16}+\frac{8\lambda h^2}{5},\\
	c &=-{1}-\frac{\lambda h^2}{10}, &d= {14}-\frac{4\lambda h^2}{5}.
\end{align*}
The absolute value of real part of $\lambda^A_{3,4}$, namely $|\Re(\lambda^A_{3,4})| $ is smaller than one if $\lambda_k h_t^2\leq 60$ and $\lambda_k h_t^2\notin[10,12]$, for $k\in \N_0$. If we introduce $\sigma(x) >0$, $x\in \Omega$, large enough, the small instability region  $\lambda_k h_t^2\in [10,12]$ would vanish and we are left with the condition  $\lambda_k h_t^2\leq 60$, for $k\in\N_0$ .

In computational examples, we often observe only the bound $\lambda_k h_t^2\leq 60$ for stability. However, there might be unstable examples where $\lambda_k h_t^2\in[10,12],$ $k\in \N_0$. Note that these estimates can also be used for the scalar wave equation when we use $S_h^2(\mathcal{T}^t)$ for the discretization since this bound applies to the ratio of the spatial eigenvalue and the time step size.

To use these insights in computational examples, we need to estimate the eigenvalues $\lambda_k$, $k\in\N_0$.  For this purpose, we use an inverse inequality for the $\textmd{curl}_x \mu^{-1}\textmd{curl}_x$ operator in the weighted $L^2_\varepsilon(\Omega)$-norm. 
For  lowest order Nédélec elements, we have the inverse inequality
\begin{align}\label{ieq:InverseInequality}
	\|\cu_x  u_h\|^2_{L^2(\mathcal{T}^x)} \leq c_I h_{\max}^{-2} \|u_h\|^2_{L^2(\mathcal{T}^x)} 
\end{align}
for all $u_h\in \mathrm{span} \{\phi_E^{\mathcal{N}^I_0}\}_{E\in\mathcal{E}}$, where
\begin{align*}
	c_I &:= \begin{cases}
		\frac{18 c_F^2}{\pi} &n=2,\\
		80 c_F^4 (\frac{9}{16\pi^2})^{2/3}&n =3,
	\end{cases}\\
	h_{\max} &:=\max_{\omega_l\in \mathcal{T}^x}( |\tau_l|^{-d} ).
\end{align*}
For the proof consider \cite[Lem.~A.2]{HauserDiss2021}. The proof is done by computing each norm on each element and showing the inequality there, then summing everything up. For completeness, we quickly state where the constants come from. The constant $c_F$ is the shape regularity constant defined in \eqref{FES:formregular}. Additionally, the coefficient $c_I h_{\max}^{-2}$ comes from estimating $18\frac{\lambda_{\max}(J_lJ_l^T)}{2\Delta_l}h_l^{-2}$ for $d=2$ and for $d=3$ estimating the term  $160\frac{4\lambda_{\max}(J_lJ_l^T)^2}{3(6\Delta_l)^2}$, where $J_l = \left(\frac{\partial F_{l,i}}{\partial \hat{x}_j}(\hat{x}_j)\right)_{1\leq i,j,\leq d}$ is the derivative of transformation $F_l : \hat{\omega}\to\omega_l$ which maps the reference element $\hat{\omega}$ to the current element $\omega_l$.

Now, we apply the inverse inequality to estimate the eigenvalues $\lambda_k$, $k\in \N_0$, from Lem.~\ref{lem:FS-H_0curl}. We rewrite the inverse inequality \eqref{ieq:InverseInequality} in the weighted norms to get
\begin{align*}
	(\mu^{-1}\cu_x  u_h,\cu_x  u_h)_{L^2(\mathcal{T}^x)} \leq c_I(\mu_{\min}\varepsilon_{\min})^{-1} h_{\max}^{-2} \|u_h\|^2_{L^2_\varepsilon(\mathcal{T}^x)} .
\end{align*}
Then we derive the following CFL condition for the vectorial wave equation
\begin{align*}
	c_I(\mu_{\min}\varepsilon_{\min})^{-1}\ h_{x}^{-2} \leq 10\ h_t^{-2},
\end{align*}
and therefore
\begin{align}
	\label{cond:CFL}
	h_t \leq \sqrt{\mu_{\min}\varepsilon_{\min}} \sqrt{\frac{10}{c_I}}h_{x},
\end{align}
where $h_t$ and $h_x$ are the respective maximal temporal and spatial step sizes from Section \ref{sec:FES}. In case $\sigma$ is big enough, e.g. $\sigma \geq 1$ everywhere, we even arrive at the CFL condition
\begin{align}
	\label{cond:CFL2}
	h_t \leq \sqrt{\mu_{\min}\varepsilon_{\min}} \sqrt{\frac{60}{c_I}}h_{x},
\end{align}

Let us consider an example for $\mu=1$ and $\varepsilon = I$. If we consider $\Omega = (0,1)^2$ and use a shape regular triangulation with isosceles rectangular triangles then we compute, by estimating the eigenvalues of $J_lJ_l^T$ directly, that $c_I=18$ as in \cite{HauserZank2023}. Hence, in this case, we arrive at the CFL conditions
\begin{align}
	\label{eq:CFL2}
	h_t & < \sqrt{\frac{60}{18}} h_x\approx 1.825741858\ h_x
\end{align}
and the stricter condition
\begin{align}
	\label{eq:CFL}
	h_t & < \sqrt{\frac{10}{18}} h_x\approx 0.74535599\ h_x.
\end{align}

Note that the computation of the CFL condition in this section can be done for an arbitrary polynomial degree in time. With a higher polynomial degree the system matrix $ A_{\text{sys}}$ will enlarge, but we still end up with conditional stability. Additionally, if we use a non-constant step size in time, simply the element matrices in our computation change and we will get a condition depending on the largest time step size. Hence, by using a tensor product approach we will always end up with a CFL condition.

  \subsection{Expected convergence rates}\label{sec:ExpectedConvRate}
	Let us take a look at the convergence rates we might expect in our examples if the CFL condition is satisfied. We will divide the analysis in two. First, we will analyze the convergence in time. Then we will discuss the convergence in space. We can split the analysis in two since we consider a tensor product structure. For this discussion we will use the concepts of the continuous solution from the proofs of Thm.~\ref{thm:ExistenceAndUniqueness}.  Again we consider $\sigma =0$ to analyze the critical case.
	
	Let us take a look at the convergence in time and  consider the semi-discrete function 
	\begin{align*}
		\vec A_{h_x}(t,x) = \sum_{k=1}^{\mathcal N_x}a_k(t) \vec\psi_{k}^{\mathcal{N}}(x).
	\end{align*}
	We will follow the steps of \cite[Sec.~5]{SteinbachZankETNA2020} and transfer the results from the scalar to the vectorial wave equation. After inserting this ansatz into the vectorial wave equation, we derive the system 
	\begin{align*}
		M_x \partial_{tt} \vec{a}(t) + A_{xx} \vec{a}(t) &= \underline{f}(t) \hspace*{.5cm}\textmd{ in } (0,T),\\
		\underline{a}(0) =\partial_{t} \underline{a}(0) &= 0
	\end{align*}
	with $\underline{a}(t) := (a_{1} (t),..,a_{\mathcal N_x}(t))^T$, where $M_x$ denotes the mass matrix and $A_{xx}$ the stiffness matrix from \eqref{FEM:Ortsmatrizen}. The right hand side $\underline{f}(t) :=(f_{1}(t),...,f_{\mathcal N_x}(t))$ is defined by $f_k(t) := ( \vec j_a(t,.), \vec\psi_k^{\mathcal{N}})_{L^2(\Omega)}$ just as in the proof of Prop.~\ref{prop:Existence_H1_drhs}. For this differential equation we derive the following variational formulation: Find $\underline{a} \in [H^1_{0,}(0,T)]^{\mathcal N_x}$ such that
	\begin{align}
		\label{varf:convrate}
		-\left(M_x\partial_t\underline{a}, \partial_t\vec W \right)_{L^2(0,T)}+\left(A_{xx}\underline{a}, \vec W \right)_{L^2(0,T)} = \left(\vec{f}, \vec W \right)_{L^2(0,T)}
	\end{align}
	for all $\underline{W} \in [H^1_{,0}(0,T)]^{\mathcal N_x}$. Next, we apply decompositions such as \cite[Prop.~47.6]{ErnGuermond2020II}. We use Eigendecomposition of the real symmetric matrix $A_{xx}$ and Cholesky decomposition for the  real, positive-definite, symmetric   mass matrix $M_{x}$  to arrive at the decompositions
	\begin{align*}
		A_{xx} =Q\Lambda Q^T \qquad \textmd{ and } \qquad M_{x} = LL^T.
	\end{align*}
	Here, $Q$ is an orthogonal matrix whose columns are the eigenvectors of $A_{xx}$ and $\Lambda$ is a diagonal matrix whose diagonal entries are the  eigenvalues of $A_{xx}$. On the other hand, $L$ is a real lower triangular matrix with positive diagonal entries.
	Next, we consider 
	\begin{align}
		\label{eqn:z_transform_u}
		\underline{z}:= Q^TL^T \underline{a}\in [H^1_{0,}(0,T)]^{\mathcal N_x}.
	\end{align}
	If we apply $L^T $ and then $Q^T$  to the equation \eqref{varf:convrate}, we derive 
	\begin{align}
		\label{eqn:discreteZ}
		-\langle \partial_t z_{j} ,\partial_t w\rangle_{L^2(0,T)} +\kappa_j \langle  z_{j} , w\rangle_{L^2(0,T)} = \langle g_j , w\rangle_{L^2(0,T)}
	\end{align}
	for  all $w\in  H^1_{ ,0} (0, T )$ and  every $j=1,...,\mathcal N_x$ with $z_{j}$ as the $j$-th entry of $\vec z$. The right hand side $g_{j}$ is the $j$-th entry of the vector  $\underline{g} =Q^TL^T\underline{f}$.
	Until now we were able to follow the same steps as \cite[Sec.~5]{SteinbachZankETNA2020}. However, in our case we have to consider $\kappa_j=0$ for $j<0$ as well. Therefore, the equation \eqref{eqn:discreteZ} represents two cases: The wave equation and the Poisson equation. Next, let us take a look at the convergence rate of each equation.
	
	If we consider $\kappa_j= 0$, we look at the Poisson problem
	\begin{align*}
		-\langle\partial_tz_j,\partial_tw\rangle_{L^2(0,T)}=\langle f,w\rangle_{(0,T)} \hspace*{.5cm} \forall w\in H^1_{,0}(0,T).
	\end{align*}
	For this  Poisson problem we know from \cite[Thm.~32.2, Lem.~32.11]{ErnGuermond2020II} that there exists  $c_1,c_2>0$  such that for all $z\in H^s(0,T)$
	\begin{align*}
		\|z_j-{z}_{h_t,j}\|_{L^2(0,T)}&\leq c_1 h_t|z_j|_{H^1(0,T)},\\
		\|z_j-{z}_{h_t,j}\|_{H^1(0,T)}&\leq c_2 h_t^{s-1}|z_j|_{H^s(0,T)},
	\end{align*}
	for $\frac{d}{2}<s\leq k+1$, where $k\geq 1$ is the degree of the Lagrange finite element. Hence, for first order elements, $k=2$, the highest expected convergence rate in the $L^2$-norm is two and in the $H^1$-norm it is one.
	
	On the other hand, if $\kappa_j>0$, we can use existing results for the wave equation \cite[Thm.~1]{SteinbachZankSpring2019}. In this theorem we consider $f\in [H^1_{,0}(0,T)]'$ and $z_j\in H^1_{0,} (0, T )$  as the unique solution of 
	\begin{align*}
		-\langle\partial_tz_j,\partial_tw\rangle_{L^2(0,T)}+\kappa_j \langle z_j,w\rangle_{L^2(0,T)} =\langle f,w\rangle_{(0,T)} \hspace*{.5cm} \forall w\in H^1_{,0}(0,T)
	\end{align*}
	satisfying $z_j\in H^1_{0,} (0, T ) \cap H^s (0, T ) $, for some $s \in[1, 2]$. Then the unique solution ${z}_{h_t,j} \in S^1_{h_t,0,} (0, T )$ of the Petrov-Galerkin  variational formulation
	\begin{align*}
		-\langle\partial_t{z}_{h_t,j},\partial_tw_{h_t}\rangle_{L^2(0,T)}+\kappa_j \langle {z}_{h_t,j},w_{h_t}\rangle_{L^2(0,T)} =\langle f,w_{h_t}\rangle_{(0,T)} \hspace*{.5cm} \forall w_{h_t}\in S^1_{h_t,,0} (0, T )
	\end{align*}
	satisfies
	\begin{align*}
		\|z_j-{z}_{h_t,j}\|_{L^2(0,T)}\leq c\left(1+\frac{4}{\pi}T^2\kappa_j\right)h_t^s\|z_j\|_{H^{s}(0,T)} + \frac{\mu Th_t^2}{6}|z_j|_{H^1(0,T)}
	\end{align*}
	with a constant $c > 0$ independent of $\kappa_j$ and $h_t$ .
	
	Next, we take a look at the convergence in space and recall convergence results for the Nédélec elements as stated in \cite[Ch.~5.5]{Monk}. Let $\tau\in\mathcal{T}_x$. By the construction of $(a_k(t))_k$ satisfying \eqref{varf:convrate} and since the vectorial wave equation is uniquely solvable, we see that $\vec A_h$ is in fact the projection of $\vec A$ in $\mathcal N_I(\mathcal{T}_x)$. Moreover, we know from \cite[Thm.~5.41]{Monk} that
	\begin{align*}
		\norm{\vec A(t,.)-	\vec A_{h_x}(t,.)}_{L^2(K)} + \norm{\cu(\vec A(t,.)-	\vec A_{h_x}(t,.))}_{L^2(K)}\leq c h_x^s (|\vec A(t,.)|_{H^s(K) ^3} +|\cu \vec A(t,.)|_{H^s(K) ^3} )
	\end{align*}
	for fixed $t\in(0,T)$ and $k\in \mathcal T_x$ if $\psi_i\in H^s(K) ^3$ and $1/2+\delta \leq s\leq k$, where $k$ is the degree of the Nédélec space. Following \cite[Rmk.~7.30]{Monk} we know that for linear edge elements on a cube we can expect all eigenfunctions $\vec \psi_i$ to be in $H^2(\Omega)$ with which we can construct $\vec A$. Therefore, we can expect first order convergence in the $H(\cu)$-seminorm.
	
	Putting all results together, we can expect for the $L^2(Q)$ norm second order convergence in time and first order in space. The second order convergence can also be shown by following the same steps as \cite[Thm.~2]{SteinbachZankSpring2019}. For the $H^{1;\cu}(Q)$-seminorm we can expect first order convergence. This we will also see in the numerical results of the next section.
	
\subsection{Numerical Results}\label{sec:linDep}
In this section, we will take a look at the numerical results and apply the theoretical results from above. We will first take a look at the results for $\sigma=0$ to have a baseline from which we can judge other results. 

To compute the  linear system \eqref{FEM:LGS}, we first have to assemble the right-hand side. The projection $\Pi_h^{\mathcal{RT},1} \vec j_a$ of the right-hand side $j_a$ in  \eqref{FES:L2RT} are calculated by using-high-order quadrature rules for the integrals, see \cite{HauserZank2023} for more details. The calculation of all spatial and temporal matrices ~\eqref{FEM:Ortsmatrizen} is done with the help of the finite element library NGSolve, see www.ngsolve.org and \cite{SchoeberlNetgen}. Finally, the linear system \eqref{FEM:LGS} is solved by the sparse direct solver  UMFPACK 5.7.1 \cite{Umfpack}.
\subsubsection{Testing the convergence for $\sigma \equiv 0$}
We will first take a look at examples for $\sigma  \equiv 0$ to investigate the CFL condition. For our examples in this section we consider the domain $Q = (0,2) \times (0,1)^2 $, with  $T=2$ and $\Omega=[0,1]\times[0,1]$, and $\varepsilon\equiv\begin{pmatrix}
	1&0\\0&1 \end{pmatrix}$ and $\mu \equiv 1$. We will consider two constructed solutions. The first is given by 
\begin{align}
	\label{Num:Lsg1}
	\vec A_1(t,x_1,x_2)=t^2x_1(1-x_1)x_2(1-x_2)\begin{pmatrix}x_2\\-x_1\end{pmatrix} 
\end{align}
for $(t,x_1,x_2) \in Q.$  The function $\vec A_1$ has homogeneous tangential trace $	\trt \vec A_1  =0$ and homogeneous initial conditions. If we insert $\vec A_1$ into the vectorial wave equation, we compute
\begin{align*}
	j_1(t,x_1,x_2) := \begin{pmatrix}
		2x_1(1-x_1)x_2(1-x_2)x_2 \\
		(2x_1(1-x_1)x_2(1-x_2)(-x_1)
	\end{pmatrix} + t^2\begin{pmatrix}
		x_1(x_1(5-12x_2)+10x_2-4)\\
		-x_2(-2x_1(6x_2-5)+5x_2-4)
	\end{pmatrix} .
\end{align*}

To take a look at the convergence rates,  we compute the experimental order of convergence (EOC) with
\begin{align*}
	EOC = \frac{\ln(err_{L-1}) - \ln(err_{L}) }{\ln(h_{L-1})-\ln(h_{L})},
\end{align*}
where $err_{L}$ is the $L^2(Q)$-error, or the  error in the $H^{\mathrm{curl};1}(Q)$-seminorm respectively, at level $L$. Additionally we use bisection in the refinement throughout the section \ref{sec:results}, hence $\ln(h_{L-1})-\ln(h_{L}) = \ln(2)$.

To compute Table \ref{tab:secOrder}, we solve the discrete linear system described in \eqref{FEM:LGS} for second-order elements in time and lowest order Nédélec elements in space. Here we see second order convergence in the $\norm{.}_{L^2(Q)}$-norm and first order convergence in the  $\abs{.}_{H^{\cu;1}(Q)}$-halfnorm. This is the same result as we would get for linear elements in time and $\Pi_h^{\mathcal{RT},1} \vec j_1$  as the projection of the right-hand side, see \cite{HauserZank2023}.

\begin{center}
		\begin{tabular}{c|c|c|c|c|c|c|c}
			L & $h_x$ & $h_t$ & \#fdofs & $\| \vec A-\vec A_h\|_{L^2(Q)} $ & EOC & $|\vec  A-\vec A_h|_{H^{\mathrm{curl};1}(Q)} $ & EOC\\ 
			\hline
			0 & 0.5000 & 0.5000 & 80 & 3.25e-02 & -  & 1.76e-01 & - \\ 
			1 & 0.2500 & 0.2500 & 896 & 1.27e-02 & 1.36 & 1.15e-01 & 0.62\\ 
			2 & 0.1250 & 0.1250 & 8192 & 3.24e-03 & 1.97 & 5.90e-02 & 0.97\\ 
			3 & 0.0625 & 0.0625 & 69632 & 7.87e-04 & 2.04 & 2.97e-02 & 0.99\\ 
			4 & 0.0312 & 0.0312 & 573440 & 1.92e-04 & 2.03 & 1.48e-02 & 1.00\\
		\end{tabular}
		\captionof{table}{Error table for the Galerkin--Petrov FEM~\eqref{vf:H1_disc} for the unit square $\Omega$ and $T=2$ and the solution $\vec A_1$ in \eqref{Num:Lsg1} using a uniform refinement strategy.}\label{tab:secOrder}
\end{center}
Next, we want to test the sharpness of the CFL condition \eqref{eq:CFL2}. To that purpose, we use a more complicated artificial solution $\vec A_2$, which was also used in \cite{HauserZank2023}. We set $T=\sqrt{2}$ to test the example of \cite{HauserZank2023} where the solution
\begin{equation}  \label{Num:Lsg2}
	\vec A_2 (t,x_1,x_2) 
	=
	\begin{pmatrix}
		-5 t^2 x_2 (1-x_2) \\
		t^2 x_1 (1-x_1)
	\end{pmatrix}+ t^3
	\begin{pmatrix}
		\sin(\pi x_1) x_2 (1-x_2) \\
		0
	\end{pmatrix}
\end{equation}
was used for $(t,x_1,x_2) \in {Q}.$ The function $\vec A_2$ fulfills the homogeneous boundary condition $	\trt \vec A_2  =0$ and has homogeneous initial conditions.
The related right-hand side $\vec j_2$ is given by
\begin{equation*}
	\vec j_2(t,x_1,x_2) = \begin{pmatrix}
		-10 (t^2-x_2^2+x_2)  \\
		2 (t^2-x_1^2+x_1)
	\end{pmatrix}
	+
	\begin{pmatrix}
		2t^3\sin(\pi x_1) +6t\sin(\pi x_1) x_2 (1-x_2) \\
		\pi t^3(1-2x_2)\cos(\pi x_1)
	\end{pmatrix}
\end{equation*}
for $(t,x_1,x_2) \in Q$ and $\mathrm{div}_x  \vec j_2 \neq 0$.

\begin{table}[h]
\begin{center}
	\caption{Errors in $\norm{\cdot}_{L^2(Q)}$ for the Galerkin--Petrov FEM  \eqref{vf:H1_disc} over $(0,\sqrt{2})\times(0,1)^2$ and the solution $\vec A_2$ in \eqref{Num:Lsg2}.}
	\begin{tabular}{c|ccccc}
			\diagbox{$h_x$}{\vspace*{-.1cm}$h_t$}&    0.2828 & 0.1414 &0.0707 &  0.0354 & 0.0177 \\
			\hline\hline
		0.1768 & 4.19e-02 & 4.19e-02 & 4.19e-02 & 4.19e-02 & 4.19e-02 \\
		0.0884 & 1.04e-02 & 1.04e-02 & 1.04e-02 & 1.04e-02 & 1.04e-02 \\
		0.0442 & 2.60e-03 & 2.08e-02 & 2.59e-03 & 2.59e-03 & 2.59e-03 \\
		0.0221 & 6.59e-04 & 1.61e-02 & 1.72e+03 & 6.47e-04 & 6.47e-04 \\
	\end{tabular}
	\label{Num:TabT2L2}
\end{center}
\end{table}

\begin{table}[h]
	\begin{center}
	\caption{Errors in $\abs{\cdot}_{H^{\cu;1}(Q)}$ for the Galerkin--Petrov FEM  \eqref{vf:H1_disc} over $(0,\sqrt{2})\times(0,1)^2$ and the solution $\vec A_2$ in \eqref{Num:Lsg2}.}
	\begin{tabular}{c|ccccc}
		\diagbox{$h_x$}{\vspace*{-.1cm}$h_t$}&   0.2828 &  0.1414 & 0.0707 & 0.0354 & 0.0177 \\
		\hline\hline		
		0.1768 & 6.22e-01 & 6.22e-01 & 6.22e-01 & 6.22e-01 & 6.22e-01 \\
		0.0884 & 3.09e-01 & 3.08e-01 & 3.08e-01 & 3.08e-01 & 3.08e-01 \\
		0.0442 & 1.54e-01 & 2.16e+00 & 1.54e-01 & 1.54e-01 & 1.54e-01 \\
		0.0221 & 7.69e-02 & 3.13e+00 & 3.61e+05 & 7.69e-02 & 7.69e-02 \\
	\end{tabular}
	\label{Num:TabT2HC1}
\end{center}
\end{table}

First let us take $T=\sqrt{2}$ and look at the $L^2(Q)$-norm error table Tab.~\ref{Num:TabT2L2}  and the $H^{\cu;1}(Q)$-seminorm error  table Tab.~\ref{Num:TabT2HC1}. This is the same example as the numerical example in the paper \cite{HauserZank2023} which showcased the sharpness of the CFL condition for piecewise linear finite elements in time and lowest order Nédélec elements in space. In the paper \cite{HauserZank2023}  we observe first-order convergence in the $H^{\cu;1}(Q)$-seminorm  and second-order convergence in the $L^2(Q)$-norm in the case of stability.
 We see  this behavior as well in the interpolation errors in Table \ref{Num:InterpT2L2} and \ref{Num:InterpT2HC1}. Here we see first-order convergence in Tab.~\ref{Num:InterpT2HC1} and second-order convergence in Tab.~\ref{Num:InterpT2L2}. When we compare these results to the results we find in the error tables \ref{Num:TabT2L2} and \ref{Num:TabT2HC1} we see first-order convergence in the $H^{\cu;1}(Q)$-seminorm and second order $L^2(Q)$-norm after stability is achieved. 

On the other hand, when we take a look at the stability of the results in table \ref{Num:TabT2L2} and \ref{Num:TabT2HC1} we clearly see the influence of the CFL condition. Following the last line we see in the third column, where $h_t/h_x \approx 0.0707/0.0221 \approx 3.1991 $ that our solution does not converge. On the other hand in the fourth column, we have $h_t/h_x \approx 0.0354/0.0221 \approx 1.6018 $ which satisfies our CFL condition and we achieve convergence. In comparison with the results for piecewise linear elements in time in \cite{HauserZank2023} we see  that the second-order elements in time  result in a smaller error in total. On the other hand, we see a quicker convergence in the time refinement when we fix the spatial discretization.

\begin{table}
	\begin{center}
	\caption{Interpolation errors in $\norm{\cdot}_{L^2(Q)}$  in $(0,\sqrt{2})\times(0,1)^2$ and the solution $\vec A_2$ in \eqref{Num:Lsg2}.}
	\begin{tabular}{c|ccccc}
		\diagbox{$h_x$}{\vspace*{-.1cm}$h_t$}&  0.2828 &  0.1414 &  0.0707 &  0.0354&  0.0177\\
		\hline\hline
		0.1768 & 7.49e-02 & 7.49e-02 & 7.49e-02 & 7.49e-02 & 7.49e-02 \\
		0.0884 & 1.93e-02 & 1.93e-02 & 1.93e-02 & 1.93e-02 & 1.93e-02 \\
		0.0442 & 4.84e-03 & 4.81e-03 & 4.81e-03 & 4.81e-03 & 4.81e-03 \\
		0.0221 & 1.24e-03 & 1.21e-03 & 1.20e-03 & 1.20e-03 & 1.20e-03 \\
	\end{tabular}
	\label{Num:InterpT2L2}
\end{center}
\end{table}

\begin{table}
	\begin{center}
	\caption{Interpolation errors in $\abs{\cdot}_{H^{\cu;1}(Q)}$  in $(0,\sqrt{2})\times(0,1)^2$ and the solution $\vec A_2$ in \eqref{Num:Lsg2}.}
	\begin{tabular}{c|ccccc}
	\diagbox{$h_x$}{\vspace*{-.1cm}$h_t$}&    0.2828&  0.1414 & 0.0707 &  0.0354 &  0.0177 \\
		\hline\hline
		0.1768 & 6.49e-01 & 6.49e-01 & 6.49e-01 & 6.49e-01 & 6.49e-01 \\
		0.0884 & 3.20e-01 & 3.20e-01 & 3.20e-01 & 3.20e-01 & 3.20e-01 \\
		0.0442 & 1.59e-01 & 1.59e-01 & 1.59e-01 & 1.59e-01 & 1.59e-01 \\
		0.0221 & 7.95e-02 & 7.94e-02 & 7.94e-02 & 7.94e-02 & 7.94e-02 \\
	\end{tabular}
	\label{Num:InterpT2HC1}	
	\end{center}
\end{table}
To showcase the behavior of the CFL condition \eqref{eq:CFL2} we take a look at the same example, but take the final time $T=\sqrt{10.4}$. We use the same number of elements in time and space as in the simulation for Tab.~\ref{Num:TabT2L2} and \ref{Num:TabT2HC1}. By taking the final time $T=\sqrt{10.4}$ we see a ratio of the time and spatial mesh size which is close to the CFL condition of section \ref{sec:cfl-condition}. For these values, we see  in both tables \ref{Num:TabL2} and \ref{Num:TabHC1}, e.g., in the third row, second to last column or last row, last column that the ratio  {equals}
\begin{align*}
	\frac{h_t}{h_x} \approx 	\frac{0.0806}{0.0442} = \frac{0.0403}{0.0221}  \approx 1.82352941.
\end{align*}
This ratio is below the CFL condition 	$h_t  < 1.825741858\ h_x.$ On the other hand in row three, column three and row four, column four we we do not satisfy the CFL condition with the ratio 
\begin{align*}
	\frac{h_t}{h_x} \approx 	\frac{0.1612}{0.0442} = \frac{0.0806}{0.0221}  \approx 3.64705.
\end{align*}
Therefore we see no convergence in our results for these cases.

\begin{table}
		\begin{center}
	\caption{Errors in $\norm{\cdot}_{L^2(Q)}$ for the Galerkin--Petrov FEM \eqref{vf:H1_disc} in $(0,\sqrt{10.4})\times(0,1)^2$ and the solution $\vec A_2$ in \eqref{Num:Lsg2}.}
	\begin{tabular}{c|ccccc}
		\diagbox{$h_x$}{\vspace*{-.1cm}$h_t$}&   0.6450 &  0.3225 &  0.1612 & 0.0806 & 0.0403 \\
			\hline\hline
		0.1768 & 3.54e-01 & 3.54e-01 & 3.54e-01 & 3.54e-01 & 3.54e-01 \\
		0.0884 & 8.90e-02 & 1.04e+00 & 8.85e-02 & 8.85e-02 & 8.85e-02 \\
		0.0442 & 2.29e-02 & 4.87e-01 & 1.95e+05 & 2.21e-02 & 2.21e-02 \\
		0.0221 & 7.88e-03 & 1.83e-01 & 9.99e+05 & 1.04e+19 & 5.52e-03 \\
	\end{tabular}
	\label{Num:TabL2}
\end{center}
\end{table}

\begin{table}
	\caption{Errors in $\abs{\cdot}_{H^{\cu;1}(Q)}$  for the Galerkin--Petrov FEM \eqref{vf:H1_disc} in $(0,\sqrt{10.4})\times(0,1)^2$ and the solution $\vec A_2$ in \eqref{Num:Lsg2}.}
	\begin{center}
	\begin{tabular}{c|ccccc}
		\diagbox{$h_x$}{\vspace*{-.1cm}$h_t$}&    0.6450 &  0.3225 &  0.1612 &  0.0806 &  0.0403 \\
			\hline\hline
		0.1768 & 4.29e+00 & 4.28e+00 & 4.28e+00 & 4.28e+00 & 4.28e+00 \\
		0.0884 & 2.15e+00 & 5.25e+01 & 2.14e+00 & 2.14e+00 & 2.14e+00 \\
		0.0442 & 1.08e+00 & 4.57e+01 & 2.00e+07 & 1.07e+00 & 1.07e+00 \\
		0.0221 & 5.57e-01 & 1.66e+01 & 1.77e+08 & 1.93e+21 & 5.36e-01 \\
	\end{tabular}
	\label{Num:TabHC1}
\end{center}
\end{table}

\subsubsection{The convergence for $\sigma \not\equiv 0$ }
As a last case, we want to take a look at the convergence if we set $\sigma \not\equiv 0$ on a subdomain of the spatial domain $\Omega$. 
For that purpose we consider again $ Q= (0,2)\times [0,1]^2$  with  $\varepsilon\equiv\begin{pmatrix}	1&0\\0&1 \end{pmatrix}$, $\mu \equiv 1$. We take the constructed solution $\vec A_1$ of \eqref{Num:Lsg1}.
However, we choose $\sigma = 1$ over  $\mathrm{conv}\{(0.5,0.35),(0.65,0.5), (0.5,0.65),(0.35,0.5)\}$ as shown in Fig.~\ref{Fig:SupportSig} and zero elsewhere. Then we compute the constructed right hand side $\vec j_a = \partial_{tt}\vec A_1+\sigma \partial_t\vec A_1 + \cu_x\cu_x\vec  A_1$.
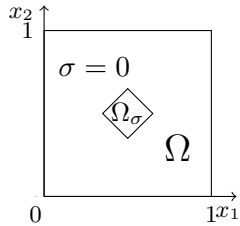
\begin{figure}[!ht]
	\centering
	\begin{tikzpicture}[scale=2.2]
		\draw (0,0) rectangle(1,1);
		\draw (0.5,0.35) -- (0.65,0.5)--(0.5,0.65)--(0.35,0.5)--cycle;
		\node (A) at (0.5,0.5) [] {\small $\Omega_\sigma $};
		\node (B) at (0.8,0.15) [above] {\Large$\Omega$};
		\node (C) at (0.3,0.9) [below] {$\sigma = 0$};
		\draw[->] (0,0) -- (1.15,0);
		\draw[->] (0,0) -- (0,1.15);
		\fill[black]  (1.1,0) circle [radius=0pt] node[below] {\footnotesize$x_{1}$}; 
		\fill[black]  (0,1.1) circle [radius=0pt] node[left] {\footnotesize$x_{2}$}; 
		\fill[black]  (-0.05,0) circle [radius=0pt] node[below] {\footnotesize$0$}; 
		\fill[black]  (1,0) circle [radius=0pt] node[below] {\footnotesize$1$}; 
		\fill[black]  (0,1) circle [radius=0pt] node[left] {\footnotesize$1$}; 
	\end{tikzpicture}
	\caption{The domain  $\Omega$ and the support $\Omega_\sigma =\text{supp}(\sigma)$ of the conductivity $\sigma.$}\label{Fig:SupportSig}
\end{figure}

When we solve \eqref{FEM:LGS} for this example we get the $L^2(Q)$-norm error as shown in Table \ref{Tab:SigmaposL2}  and the $H^{\cu;1}(Q)$-seminorm error shown in Table \ref{Tab:SigmaposH1C} for piecewise quadratic functions in time. In both tables, Tab.~\ref{Tab:SigmaposH1C} and Tab.~\ref{Tab:SigmaposL2}, we see the CFL condition \eqref{eq:CFL2}, $h_t  < 1.825741858\ h_x$, as in the case of $\sigma \equiv 0$. Indeed, in each table, we see in the last row and the second column where $h_t/h_x = 4.7348 $ to the third column where $h_t/h_x = 2.3674 $ an increase in the error which stops after the CFL condition is satisfied in the fourth column where $h_t/h_x=1.18 $. After the CFL condition is met we again see linear convergence in the $H^{\cu;1}(Q)$-seminorm but no second-order convergence in the $L^2(Q)$-norm, however. Note, that the stricter CFL condition \eqref{eq:CFL} does not apply here.

\begin{table}[!ht]
	\caption{Errors in $\norm{\cdot}_{L^2(Q)}$ for the Galerkin--Petrov FEM~\eqref{vf:H1_disc} in $(0,{2})\times(0,1)^2$ and the solution $\vec A_1$ in \eqref{Num:Lsg1}.}
	\begin{center}
	\begin{tabular}{c|ccccc}
			\diagbox{$h_x$}{\vspace*{-.1cm}$h_t$}&   0.2500 &  0.1250 &  0.0625 &  0.0312 &  0.0156 \\
			\hline\hline
		0.2111 & 1.06e-02 & 1.06e-02 & 1.06e-02 & 1.06e-02 & 1.06e-02 \\
		0.1055 & 3.72e-03 & 3.70e-03 & 3.70e-03 & 3.70e-03 & 3.70e-03 \\
		0.0528 & 2.23e-03 & 1.31e+00 & 2.14e-03 & 2.14e-03 & 2.14e-03 \\
		0.0264 & 1.49e-03 & 1.77e+01 & 6.57e+10 & 1.48e-03 & 1.48e-03 \\
	\end{tabular}\label{Tab:SigmaposL2}
	\end{center}
\end{table}

	\begin{table}[!ht]
	\caption{Errors in $\abs{\cdot}_{H^{\cu;1}(Q)}$ for the Galerkin--Petrov FEM~\eqref{vf:H1_disc} in $(0,{2})\times(0,1)^2$ and the solution $\vec A_1$ in \eqref{Num:Lsg1}.}
	\begin{center}
	\begin{tabular}{c|ccccc}
		\diagbox{$h_x$}{\vspace*{-.1cm}$h_t$}&   0.2500 & 0.1250 &  0.0625 &  0.0312 &  0.0156 \\
			\hline\hline
		0.2111 & 1.08e-01 & 1.08e-01 & 1.08e-01 & 1.08e-01 & 1.08e-01 \\
		0.1055 & 5.76e-02 & 5.47e-02 & 5.47e-02 & 5.47e-02 & 5.47e-02 \\
		0.0528 & 6.48e-02 & 1.49e+02 & 2.76e-02 & 2.76e-02 & 2.76e-02 \\
		0.0264 & 2.12e-02 & 3.42e+03 & 1.49e+13 & 1.39e-02 & 1.42e-02 \\
	\end{tabular}\label{Tab:SigmaposH1C}
\end{center}
\end{table}

\begin{rmk}
	In the case of $\varepsilon = \varepsilon_0$,  $\mu^{-1} = \mu_0^{-1}$ we get 
	$	h_t  <\sqrt{ \varepsilon_0\mu_0}\sqrt{\frac{60}{18}} h_x\approx 6.09	\ 10^{-9 } \ h_x$ 
	since  $\mu_0 = 1.256637\ 10^{-6}$ and $  \varepsilon_0 = 8.854188\ 10^{-12}.$ In this case, since $\mu^{-1}$ and $\sigma$ are at least of order $10^{12}$ greater, it might be reasonable to think about the application of this problem. Either, we consider a very fine resolution where we are interested in the resolution on the atomic level or it might be reasonable to consider the Eddy Current problem instead of solving \eqref{vf:H1_disc}. This will be a topic of future work.
\end{rmk}

\section{Conclusion} \label{Sec:Zum}	
In this paper, we have investigated the unique solvability of the variational formulation of the vectorial wave equation considering Ohm's law. First, we have taken a look at the trial and test spaces and showed properties of the functional spaces that we needed to prove the solvability. Then we proved unique solvability for the variational formula \eqref{vf:H1VarForm}, where $\vec j_a\in L^1(0,T;L^2(\Omega;\R^d))$. For electromagnetic problems, the variational formula \eqref{vf:H1VarForm} applies to a variety of electromagnetic examples since it is posed for Ohm's law and anisotropic material.

Having proven the unique solvability, we turned to computational examples in the tensor product. We used piecewise quadratic ansatz functions in time and lowest order Nédélec elements in space. Here, as in the case of linear ansatz functions in time, \cite{HauserZank2023}, we realized that there is a CFL condition. We calculated the CFL condition and gave examples. Using the reasoning from Section 3.3, it is also possible to derive a CFL condition for other higher-order elements in time. We thus learn that simply increasing the order of finite elements in time does not lead to an unconditionally stable method. In the case of the tensor product structure, we can always expect a CFL condition for the space-time discretization of the vectorial wave equation. A possible solution is the use of the modified Hilbert transform as performed in \cite{HauserZank2023}. 

The results of this work form the basis for more complicated electromagnetic problems. We observed
the main difficulties of the vectorial wave equation and what they imply. This will be a starting point for future work on the calculation of eddy current problems and further applied calculations such as unstructured space-time meshes.
\bibliographystyle{acm}
\bibliography{references}

\end{document}